\documentclass{amsart}

\usepackage{color}
\usepackage{amssymb,amsfonts,amsmath,amsthm}
\usepackage{tikz}
\usetikzlibrary{bayesnet}
\usepackage{bbm}

\usepackage{enumitem}
\usepackage{hyperref}
\usepackage{float}
\usepackage{mathtools}

\newcommand{\R}{\mathbb{R}}

\newcommand{\N}{\mathbb{N}}

\newcommand{\mX}{\mathcal{X}}

\newcommand{\mG}{\mathcal{G}}

\newcommand{\mA}{\mathcal{A}}
\newcommand{\mB}{\mathcal{B}}
\newcommand{\mN}{\mathcal{N}}

\newcommand{\PP}{\mathbb{P}}

\usepackage{geometry}
\newcommand\red[1]{\textcolor{red}{#1}}

\newtheorem{theorem}{Theorem}

\newtheorem{lemma}[theorem]{Lemma}
\newtheorem{corollary}[theorem]{Corollary}

\theoremstyle{definition}
\newtheorem{definition}[theorem]{Definition}
\newtheorem{remark}{Remark}
\newtheorem{example}{Example}

\newtheorem{claim}{Claim}

\begin{document}

\title[Reduced BNs]
  {Reductions of discrete Bayesian networks via lumping}
\author{Linard Hoessly}

\address{Data Center of the Swiss Transplant Cohort Study,
University hospital Basel, Basel,
4031 Switzerland.}
\email{Linard.hoessly@hotmail.com {\rm (corresponding author)}}


\date{\today}
\maketitle
\begin{abstract}

Bayesian networks are widely utilised in various fields, offering elegant representations of factorisations and causal relationships. We use surjective functions to reduce the dimensionality of the Bayesian networks by combining states and study the preservation of their factorisation structure. We introduce and define  corresponding notions, analyse their properties, and provide examples of highly symmetric special cases, enhancing the understanding of the fundamental properties of such reductions for Bayesian networks. We also discuss the connection between this and reductions of homogeneous and non-homogeneous Markov chains. 
\end{abstract}

\maketitle

\section{Introduction}
Bayesian networks (BNs) based on directed acyclic graphs (DAGs) are a  popular statistical model, where cause and effects can directly be read off the DAG. They are applied across fields like economics, social sciences, computational biology, or computer science \cite{Wainwright06}. Their factorisation structure simplifies the treatment of the corresponding probabilities both computationally and notation-wise. Given a BN with a DAG $\mG=(V,E)$ with vertices $V$ and edges $E$, the BN associates to each vertex $v$ a discrete random variable $X_v$, where the arrows of $\mG$ capture the conditional independence (CI) assumption of the underlying probability distribution as follows,
$$
X_i \perp\!\!\!\perp X_{nd(i)\setminus pa(i)}|X_{pa(i)},
$$
where $\cdot\perp\!\!\!\perp\cdot | \cdot$ is the usual CI notion from probability theory \cite{koller2009probabilistic}, where $pa(i)$, $nd(i)$ are the parents and non-descendants of $i$ (for more details see $\S$ \ref{subs_DAGs}). The CI structures are equivalent to factorisations of the conditional probablity distributions (CPDs) along the DAG \cite{koller2009probabilistic}.

Next, we explain our setting of interest. Consider a BN as a discrete random vector $(X_v)_{v\in V}$ where for each $v\in V$, $X_v$ has the same identical states $\mA$. Consider a surjective function $f:\mA\to \mB$ that maps $\mA$ to some smaller set of states $\mB$. Let the
BN have DAG $\mG$ with corresponding factorisation of the discrete probability distribution. We introduce three situations of interest for a given BN and a function $f$
such that the reduced random vector $(U_v)_{v\in V}:=(f(X_v))_{v\in V}$ also factorises via a BN with the same DAG:

\begin{enumerate}[label=\emph{(D\arabic*)}]
\item The CI structure is preserved for $(U_v)_{v\in V}$.
\item The CI structure is preserved for all possible initial distributions (also called prior distribution \cite{koller2009probabilistic}, i.e. the distribution on the source nodes). \color{black}
\item The CI structure is preserved for all possible initial distributions such that independently of the initial distribution the CPDs of $(U_v)_{v\in V}$ are the same.
\end{enumerate}
The above notions with examples are further  explained in $\S$ \ref{sect_def_lump_GM}.
\color{black}

\hfill\break

The above notions are motivated by different applications and research directions. One interest is analytically simplifying probabilistic models, i.e., creating coarse-grained models that reduce complexity while preserving the factorisation (or CI) structure. This can be beneficial for inference tasks since learning BNs can be computationally challenging, e.g., many related learning tasks are NP-hard in general \cite{scutari2021bayesian,wainwright2008}. By choosing a simplification approach, we can proceed in different stages by gradually increasing complexity.

On the other hand, merging categorical variables might be of interest when the desired level of information is coarser than the available data. In such cases, one implicitly assumes that the corresponding function of the variables preserves the CI structure. This assumption has philosophical implications since in many fields (cf., e.g., \cite{koller2009probabilistic}), practitioners begin by drawing a directed acyclic graph (DAG) to outline the causal structure. In practice, the states of random variables are often summarized in some way (e.g., as in  example \ref{well-known-sprinkler}). This indicates that when we assume the DAG assumption holds for more precise states, we are actually working within a restricted subspace. In particular, often when we apply a discrete BN to a real-world example we treat it as an idealisation, where states are merged. The conditions developed in our research can be utilised to clarify such properties.

From another viewpoint, in probability theory reductions preserving the Markov property for Markov chains have been studied since the Fifties \cite{Burke_Rosenblatt,kemenyfinite,DEY2014,gurvits}. Our contribution generalises and connects to these fields as well.

\subsection{Previous approaches}\label{prev_appr}
We explore a dimensionality reduction approach for BNs while preserving their graphical structure. Although this has not been studied, it is related to BNs with hidden or latent variables. In such cases, certain vertices in the network are observed only in a merged form, which has been extensively investigated by Elidan and Friedman \cite{Elidan1,Elidan2,Elidan3} under the term BNs with hidden variables.
When a variable is entirely hidden, it can be viewed as a projection onto a single state. The authors assume that an observed BN is a projection where states are potentially merged \cite{Elidan1}, and their objective is to reconstruct information on the original BN. Furthermore, reductions from continuous to discrete are used as a strategy to learn the BN in \cite{Moffa_ordinal}, where projections to finite ordinary data help with the inference. However, our objective differs from the previously mentioned works, as we aim to study and establish conditions that guarantee the preservation of the factorisation structure during the dimensionality reduction process.\\

Another related and similar field are reductions of Markov chains.
Simplifying discrete-time Markov chains (DTMCs) via reduction functions while maintaining their Markov property is a fundamental procedure going back to the work of Burke and Rosenblatt from 1958 \cite{Burke_Rosenblatt}. These simplifications involve reducing the complexity of DTMCs trough state space reduction while ensuring that the resulting process remains Markovian. In the context of BNs, DTMCs can be represented in the following form:\\
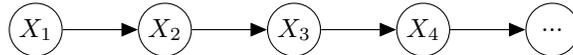
\begin{center}
\begin{tikzpicture}
    \node[latent] (theta) {$X_1$} ; %
    \node[latent, right=of theta] (z) {$X_2$} ; %
    \node[latent, right=of z] (y) {$X_3$} ; %
    \node[latent, right=of y] (mu) {$X_4$} ; %
    \node[latent, right=of mu] (sigma) {...} ; %
    \edge {theta} {z} ; %
    \edge {z} {y} ; %
      \edge {y} {mu} ; %
            \edge  {mu}{sigma} ; %
\end{tikzpicture}
\end{center}
Preserving the Markov structure of DTMCs corresponds to maintaining the factorisation structure in BNs with fixed same CPDs for all edges \cite{Burke_Rosenblatt,kemenyfinite,DEY2014,gurvits}. For DTMCs, the so-called lumped chain is the process obtained when projecting the states onto a partition. However, the lumped process generally loses the Markov property  \cite{Burke_Rosenblatt}. There are two kinds of lumpability with respect to a lumping function for DTMCs:\begin{enumerate}
    \item Weak lumpability: A DTMC is weakly lumpable if the lumped chain itself is a DTMC.
    \item Strong lumpability: A DTMC is strongly lumpable if, for all initial distributions, the lumped chain is a DTMC.

    \end{enumerate}
 Lumpings find practical applications in fields requiring dimension reductions of DTMCs such as complex networks \cite{Weinan_cx}, control theory or related fields \cite{gurvits}.Recent advancements have focused on studying lumpability for non-homogeneous DTMCs (NHDTMCs) \cite{DEY2014}, broadening its applicability. Furthermore, researchers have explored the application of lumpability to complex networks, investigating its implications and developing "best approximate lumpings" for DTMCs . Remark that weak lumpability is similar to \ref{D1}, while strong lumpability is similar to \ref{D2} and \ref{D3}, but for DTMCs. In particular, weak or strong lumpability are less general notions, and a DTMC satisfies \ref{D1} as a BN if and only if the lumped process is an NHDTMC. 

\subsection{Content}
In $\S$\ref{sec_prel}, we introduce notions for BNs, reductions of BNs and and look at DTMCs/NHDTMCs as BNs. In $\S$\ref{sect_Res}, we give results on conditions ensuring that reductions preserve the factorisation structure of the DAG. In $\S$\ref{sec_ex_cases}, we go through some special cases, and in $\S$\ref{sec_disc} we discuss the reductions and context.

\subsection*{Acknowledgements}
We thank Giusi Moffa, Ioan Manolescu, Christian Mazza, and Ulrich Hansen for helpful discussions.\\
\color{black}
\section{Preliminaries}\label{sec_prel}

\subsection{Graph theory and notation for DAGs}\label{subs_DAGs}
 Let $\mG=(V, E)$ be a finite directed graph \cite{Diestel}. 
 For directed paths, the length of paths is the number of edges in the path. A cycle is a directed path where the starting vertex of the first edge equals the ending vertex of the last edge, i.e. $v_{i_1}\rightarrow v_{i_2}\rightarrow \cdots \rightarrow v_{i_{l-1}}\rightarrow v_{i_{l}}=v_{i_1}$ (potentially listed as the corresponding sequence of edges). A directed acyclic graph (DAG) is a directed graph that has no cycles. Furthermore, we let the indegree and outdegree of a vertex $v\in V$ be the number of incoming resp. outgoing edges. For $W\subseteq V$ let the vertex-induced subgraph be the graph with vertices $W$ and edges all edges in $E$ with head and tail in $W$, i.e. edges $E_W$, with $E_W:=\{e\in E|e=(x,y) \text{ and }x,y\in W\}$.\\
 
We will use the following sets of vertices associated to a vertex $v\in V$ of a DAG:\\
- $pa(v)$ is the set of parents of $v\in V$, i.e. all $w\in V$ such that $(w,v)\in E$.\\
- $de(v)$ is the set of descendants of $v\in V$, i.e. all nodes $w\in V$ such that there is a directed path $v\to v_1 \to \cdots \to w$ 
that goes from $v$ to $w$.\\
- $dde(v)$ is the set of direct descendants of $v\in V$, i.e. all nodes $w\in V$ such that there is a directed edge $v\to  w$.\\
- $nd(v)$ is the set of non-descendants of $v\in V$, which are given by $V\setminus (\{v\}\cup de(v))$.\\
 - $pr(v)$ is the set of predescessors of $v\in V$, i.e. all nodes $w\in V$ such that there is a directed path $w\to w_1\to \cdots v$ 
  that goes from $w$ to $v$.\\
  - Furthermore, we denote by $\cdot^\ast(v)$ the sets of vertices including $v\in V$, i.e.,
$$pr^*(v):=\{v\}\cup pr(v), \quad pa^*(v)=\{v\}\cup pa(v), \quad de^*(v)=\{v\}\cup de(v), \quad nd^*(v)=\{v\}\cup nd(v).$$
Let $V_p$ be the vertices with parents $V_p:=\{v\in V| pa(v)\neq \emptyset\}$, $V_s$ the so-called source nodes, i.e. $V_s:=\{v\in V| pa(v)= \emptyset\}$, and $V_{no-d}$ the so-called sink nodes, i.e. $V_{no-d}:=\{v\in V| de(v)= \emptyset\}$.

Furthermore, we will need the following notion for DAGs. Let $\mG$ be a DAG, and $v\in V$ be a vertex. Then, the depth of v is defined as
$$depth(v):=max\{k\in \N\ |\ \text{there is a } w\in V_s\text{ and a directed path from w to v of length }k\}$$
For $A\subseteq V$ a subset of the vertices we define $depth(A):=max \{ depth(v)|v\in A\}$.

\subsection{Notations}\hfill\break
For $C$ a finite set, we let $\Delta^C$ be the probability simplex in $\R^C$, and we denote by $\N$ the positive integers, and by $\N_0$ the nonnegative integers, i.e. $\N_0:=\N\cup \{0\}$.\\

Let $G=(V,E)$ be the DAG of the BN with random vector $(X_v)_{v\in V}$. 
Let
 $\mA= \{a_1,\cdots a_k\}$ denote the possible states in each vertex such that $(X_v)_{v\in V}$ takes values in $\mA^V$. If we emphasise the index, we denote the states of $X_v$ by $\mA^v= \{a^v_1,\cdots a^v_{k_v}\}$ for $v\in V$.
Let $f:\mA\to\mB$ be the reduction map, i.e. a surjection that maps some states together. Denote the cardinalities by $|\mA|=k>l=|\mB|$. This induces an equivalence relation on $\mA$ such that $\overline{a_i}=\{a_j\in\mA|f(a_i)=f(a_j)\}$, and we can identify $\overline{a_i}$ with $f(a_i)$ such that $\mB\simeq \mA/\sim$ (bijection).
Let $\mB =\{b_1,\cdots b_l\}$ be the reduced states such that $(f(X_v))_{v\in V}$ takes values in $\mB^V$. Denote by $(U_v)_{v\in V}=(f(X_v))_{v\in V}$ the projected random vector.
For a subset $W\subseteq V$ with $W=\{v_1,\cdots v_k\}$, and $a_W\in \mA^W$ we denote
$$ \PP(X_W=a_W):=\PP(x_{v_1}=a_{v_1},\cdots, x_{v_k}=a_{v_k}).$$
Furthermore, to simplify notation when the context is clear, we will allow the following abuse of notation,\color{black} where for subsets $\tilde{W}\subseteq W\subseteq V$ and for $a_W\in\mA^W$ we might denote the restriction of $a_W$ to the coordinates of $\tilde{W}$ by $a_W |_{\tilde{W}}$ or $a_{\tilde{W}}$.

\subsection{Bayesian networks}
We restrict  our treatment to random variables with discrete state space \cite{koller2009probabilistic}. 
Let $(X_v)_{v\in V}$ be a random vector indexed by $V$ taking values in the product space $\mX=\otimes_{v\in V}\mX_v$, and $\mG=(V,E)$ a DAG. Then, a (discrete) distribution factorises according to $\mG$ if 
\begin{equation}
\label{factorisation_BN}
\PP((X_v)_{v\in V})=\prod_{v\in V}\PP(X_v|X_{pa(v)}),
\end{equation}
where $\PP(X_v|X_{pa(v)})$ are the CPDs.
A BN is a pair $\mB\mN=(\mG,\PP)$ where $\PP$ factorises over $\mG$, and correspondingly one says $(X_v)_{v\in V}$ factorises w.r.t $\mG$. Then,

  the so-called prior distribution (or starting distribution) is the distribution of the random variables associated to vertices of $V_s$, i.e., $(\alpha_v)_{v\in V_s}\in\prod_{v\in V_s}\Delta^{\mA^v}$.
We will say that $(X_v)_{v\in V}$ has full support if for any possible state $x\in \mX$, the probability that $(X_v)_{v\in V}$ takes this value is nonzero.\color{black}

For a triple $(A,B,S)$ of disjoint subsets of $V$ we say that $S$ d-separates $A$ from $B$ in $\mG$ if for undirected path between a vertex of $A$ and $B$, there is a vertex $v$ in the path such that
\begin{itemize}
\item $v\in S$ and the edges of the path do not meet in $v$ head-to-head.
\item $v\not\in S$ nor any of its descendants, and the edges meet head-to-head in $v$.
\end{itemize}

The above factorisation is equivalent to both the local and the global Markov properties of a distribution with respect to the DAG $\mG$ as follows.
\begin{theorem}\label{equivalence_local_MP}\cite[Theorem 3.27]{lauritzen1996}
Let $(X_v)_{v\in V}$ be a discrete random vector. The following are equivalent:\\
- $(X_v)_{v\in V}$ factorises over the DAG $\mG$.\\
- (Local Markov property) For any $v\in V$, $(X_v)_{v\in V}$ satisfies
$$
X_v \perp\!\!\!\perp X_{nd(v)\setminus pa(v)}|X_{pa(v)}.
$$
- (Global Markov property) For any triple $(A,B,S)$ of disjoint subsets of $V$ such that $S$ d-separates $A$ from $B$ in $\mG$, $(X_v)_{v\in V}$ satisfies 
$$
X_A \perp\!\!\!\perp X_{B}|X_S.
$$ 
\end{theorem}

If we consider all possible BNs for a given DAG structure $\mG$, e.g., as in \cite{Allmanid}, the CPDs $(\PP(X_v|X_{pa(v)}))_{v\in V}$ act as parameters of the DAG model, such that we denote the space of all possible parameters as $\Theta$. Then, the parametrization map is a map
$$\theta:\Theta\to \Delta^{(\prod_{v\in V}n_v)-1},$$
where $n_v$ is the number of elements of the discrete state space of $X_v$ for $v\in V$  \cite{Allmanid}. In our setting for $\mA$ resp. $\mB$ $n_v$ equals $k$ resp. $l$. 
Also, we write $im(\mG):=\theta(\Theta)$, which is typically a proper subset of $\Delta^{(\prod_{v\in V}n_v)-1}$, and for state spaces $\mA$ fix we write $im(\mG,\mA)$ (and for $\mB$ accordingly).
We also recall that while two Markov-equivalent DAGs $\mG_1,\mG_2$ might have different parametrisation maps, $im(\mG_1)=im(\mG_2)$.

Furthermore, in the following we assume the BN $\mB\mN$ has full support. This such that, in accordance with Markov theory when we consider a BN $\mB\mN=(\mG,\PP)$ that represents the random variable $(X_v)_{v\in V}$ we can replace the original initial distribution $(\alpha_v)_{v\in V_s}\in\prod_{v\in V_s}\Delta^{\mA^v}$ and replace it by another initial distribution $(\tilde{\alpha}_v)_{v\in V_s}\in\prod_{v\in V_s}\Delta^{\mA^v}$. We denote this by its expectation operator $\mathbb{E}_{\tilde{\alpha}}(\cdot)$, its probability by $\PP_{\tilde{\alpha}}(\cdot)$, by writing the BN as $(X_v[\tilde{\alpha}])_{v\in V}$, which we call the $(\tilde{\alpha}_v)_{v\in V_s}$-changed random variable.
\begin{remark}\label{full_supp_rem1}
Note that for the above we need all CPDs to be well-defined, hence we assume full support.
\end{remark}
 Furthermore, for a given BN $\mB\mN=(\mG,\PP)$ with states in $\mA$, we let $im(\mG,\mB\mN)|_{\prod_{v\in V_s}\Delta^{\mA^v}}$ denote the restriction of the image of all BNs when we only change initial distributions.

\subsection{Lumping of directed graphical models}\label{sect_def_lump_GM}
Consider a BN $\mB\mN=(\mG,\PP)$ with discrete random vector $(X_v)_{v\in V}$ with full support. Assume each random variable $X_v$ has the same possible states $\mA$ and consider a surjective function $f$ that maps $\mA$ to some smaller set of states $\mB$.  

We are interested in conditions on the reduction function $f$ such that the random vector $(U_v)_{v\in V}:=(f(X_v))_{v\in V}$ has a factorisation for the DAG $\mG$. 
\begin{definition}\label{def_Di}
We distinguish three situations for the BN $\mB\mN=(\mG,\PP)$ with $(X_v)_{v\in V}$ and reduction function $f:\mA\to\mB$:
\begin{enumerate}[label=\emph{(D\arabic*)}]
\item\label{D1} $(U_v)_{v\in V}$ factorises over the DAG $\mG$.
\item\label{D2} For any initial distribution $\tilde{\alpha}
\in\prod_{v\in V_s}\Delta^{\mA^v}$,

$(U_v[\tilde{\alpha}])_{v\in V}$ factorises over the DAG $\mG$.

 \color{black}
\item\label{D3} For any initial distribution $\tilde{\alpha}\in\prod_{v\in V_s}\Delta^{\mA^v}$, 
$(U_v[\tilde{\alpha}])_{v\in V}$ factorises over the DAG $\mG$ with the same CPDs (apart from initial distributions) independently of $\tilde{\alpha}$.

\end{enumerate}
Correspondingly, we will say that \ref{D1} (resp.  \ref{D2}, \ref{D3}) holds for $(\mB\mN,f)$.
\end{definition}

\begin{example} \label{well-known-sprinkler}To start, consider the well-known example of a BN with the DAG
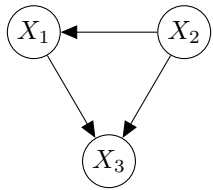
\begin{center}
\begin{tikzpicture}
 \node[latent] (x) {$X_3$};%
 \node[latent,above=of x,xshift=-1cm,fill] (y) {$X_1$}; %
 \node[latent,above=of x,xshift=1cm] (z) {$X_2$}; %

 \edge {y,z} {x}  
  \edge {z} {y}  
 \end{tikzpicture}

\end{center}where $X_1$ stands for \textbf{sprinkler}, $X_2$ for \textbf{rain}, and $X_3$ for \textbf{wet grass}.
Imagine starting not only with the clear zero/one resp. on/off states, but e.g., 3 states as 0-1mm/h, 1-3 mm/h,  and more than 3 mm/h (resp. 0-1mm, 1-3 mm,  and more than 3 mm for $X_3$). 
Then we map these states by always joining the last two, giving clear zero/one resp. on/off states corresponding to a reduction as in Definition \ref{def_Di}.
\end{example}

\begin{remark}\label{full_supp_rem_wrt_red}
Note that \ref{D1} is completely independent from the assumption of full support. \ref{D2}/\ref{D3} only need that the corresponding CPDs are defined. Hence, whenever all CPDs are defined, \ref{D2}/\ref{D3} are defined equivalently. For the convenience of the reader we assume full support as a natural setting such that the corresponding results apply also to random vectors that factor w.r.t. a DAG.\end{remark}

\begin{example}\label{simple_ex}
Consider the BN $\mB\mN=(\mG,\PP)$ $(X_v)_{v\in V}$ with DAG $\mG:=v_1 \to v_2$ with $\mA=\{a_1,a_2,a_3\}$, $\mB=\{b_1,b_2\}$, and $f:\mA\to \mB$ with $f(a_1)=f(a_2)=b_1,f(a_3)=b_2$ and CPD
$$(\PP(X_2=x_2|X_1=x_1))_{x_1,x_2}=\begin{pmatrix} 1/2& 1/2 &0 \\ 0 &1/2 &1/2\\1/2& 1/2& 0 \end{pmatrix},$$
where for the moment we do not consider fix initial distribution. We note the following:
\begin{itemize}
\item For any initial distribution, the reduction $(f(X_v))_{v\in V}$ factorises w.r.t. the same DAG.
\item Hence, $(\mB\mN,f)$ satisfies both \ref{D1} and \ref{D2}.
\end{itemize}
In order to show that \ref{D3} does not hold,
 consider the parametrisation of $\Delta_2$ on $\Delta_{\mA^{v_1}}$ via the map
$$\begin{pmatrix} x \\ y \end{pmatrix} \to \begin{pmatrix} xy \\ x(1-y) \\ 1-x \end{pmatrix}.$$
As we map the first two states to $b_1$, clearly this factorises, and the following holds for $0<x\leq1$:
$$\PP(Y_2=b_1|Y_1=b_1)=\frac{\PP(Y_2=b_1,Y_1=b_1)}{\PP(Y_1=b_1)}=\frac{xy+1/2x(1-y)}{x}=1/2+1/2y,$$
which depends on the initial distribution. Hence, \ref{D3} does not hold.
\end{example}

We have the following implications for the notions.
\begin{lemma}\label{strength_Ds}
Consider a BN $\mB\mN=(\mG,\PP)$ with random vector $(X_v)_{v\in V}$ and reduction function $f:\mA\to\mB$. Then, the following implications hold for $(\mB\mN,f)$:\\
$$\ref{D3}\implies \ref{D2}\implies \ref{D1}$$
Furthermore, none are equivalences.
\end{lemma}
\begin{proof}
The implications follow by definition as we add more and more conditions. That \ref{D1} is not equivalent to \ref{D2} follows by a BN that satisfies \ref{D1} but not \ref{D2}, which is clear by example \ref{ex_D1_not_D2}. Similarly, the lack of equivalence between \ref{D2} and \ref{D3} follows, e.g., from example \ref{simple_ex}.
\end{proof}

 \subsection{DTMCs as BNs}
DTMCs are sequences of random variables $(X_n)_{n\in\N}$ with values in a state space $\mA$ (usually finite or countable), where the probability to transition from one state to another only depends on the current state \cite{norris}, detemined by the transition matrix $P$. DTMCs as a BNs are illustrated in $\S$ \ref{prev_appr}. Note however that for DTMCs all CPDs are always the same.
Higher-order DTMCs can also be considered as BNs, e.g., an order two DTMC has transition probabilities that depend on the last two previous states, giving the following DAG.\\

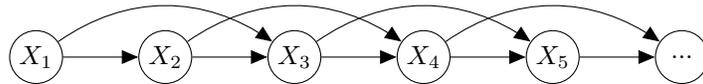
\begin{center}
\begin{tikzpicture}
    \node[latent] (theta) {$X_1$} ; %
    \node[latent, right=of theta] (z) {$X_2$} ; %
    \node[latent, right=of z] (y) {$X_3$} ; %
    \node[latent, right=of y] (mu) {$X_4$} ; %
    \node[latent, right=of mu] (sigma) {$X_5$} ; %
        \node[latent, right=of sigma] (sigm2) {...} ; %
    \edge {theta} {z} ; %
        \edge[bend left=35] {theta} {y} ; %
    \edge {z} {y} ; %
            \edge[bend left=35] {z} {mu} ; %
      \edge {y} {mu} ; %
                  \edge[bend left=35] {y} {sigma} ; %
            \edge  {mu}{sigma} ; %
                              \edge {sigma} {sigm2} ; %
            \edge[bend left=35]  {mu}{sigm2} ; %
\end{tikzpicture}
\end{center}
NHDTMCs are similar in that transition probabilities of states depend only on the current state, but are time-dependent and given by a sequence of transition matrices $(P_n)_{n\in\N}$. In particular their DAG structure remains the same as the one of DTMCs illustrated in $\S$ \ref{prev_appr}.
Therefore, a DTMC as a BN satisfies \ref{D1} iff the lumped process is a NHDTMC. Some consequences and examples are given in $\S$ \ref{sec_ex_cases_DTMCs} for DTMCs and in $\S$ \ref{subsec_NHDTMCs} for NHDTMCs.

\section{Results}\label{sect_Res}
We focus on random variables with the same underlying space $\mA$ 
and reduce them via a map $f:\mA\to\mB$.
We start with general considerations on reductions in $\S$ \ref{General_red_sect}, and then give conditions on BNs for \ref{D1}, \ref{D2}, and \ref{D3} in $\S$ \ref{res_D1}, $\S$ \ref{res_D2}, and $\S$ \ref{res_D3}.
\subsection{Generalities on projected BNs}\label{General_red_sect}
Let $\hat{f}(\cdot)$ denote the image measure induced on the probability distributions through $f$ (cf., e.g., \cite[Definition 7.7]{schilling2017measures}). First, we can reformulate \ref{D1},\ref{D2} and \ref{D3} by connecting it to the set of all BNs of a DAG $\mG$ on states $\mB$.
\begin{lemma}\label{conn_Markov}
Consider a BN $\mB\mN=(\mG,\PP)$ with full support, random vector $(X_v)_{v\in V}$ and a projection $f:\mA\to \mB$, with $(U_v)_{v\in V}:=(f(X_v))_{v\in V}$. Let $z\in \Delta^{\prod_{v\in V}\mA}$ equal the distribution of $\mB\mN$.Then,
\begin{itemize}
\item \ref{D1} holds for $(\mN\mB,f)$ if and only if $\hat{f}(z)\subseteq im(\mG,\mB)$
\item \ref{D2} holds for $(\mN\mB,f)$ if and only if $\hat{f}(im(\mG,\mB\mN)|_{\prod_{v\in V_s}\Delta^{\mA^v}})\subseteq im(\mG,\mB)$
\item \ref{D3} holds for $(\mN\mB,f)$ if and only if $\hat{f}(im(\mG,\mB\mN)|_{\prod_{v\in V_s}\Delta^{\mA^v}})\subseteq im(\mG,\mB)$ and the nontrivial CPDs in the image of $\hat{f}(im(\mG,\mA)|_{\prod_{v\in V_s}\Delta^{\mA^v}})$ are always the same.
\end{itemize}
\end{lemma}

Reductions of BNs typically lose their factorisation in the following sense. Recall that Markov equivalence classes of BNs are characterised through the skeleton and immoralities of the DAG (cf., e.g., \cite[Theorem 3.8]{koller2009probabilistic} or \cite{frydenberg}).  Then, for many DAGs $\mG$ and any nontrivial reduction function there are always some BNs whose reduction does not factorise, i.e. where $(\mN\mB,f)$ does not satisfy \ref{D1}. Hence, by Lemma \ref{strength_Ds} neither do \ref{D2} or \ref{D3} hold. The proof is in the Appendix $\S$ \ref{proof_many_non_D1s}.

\begin{theorem}\label{many_non_D1s}
Consider a connected DAG $\mG=(V,E)$, $|V|\geq 3$ which has at least one connected induced subgraph on three vertices whose skeleton is not a complete graph and which is not a v-structure. Consider a surjective reduction function $f:\mA\to \mB$ with $|\mA |>|\mB |>1$. Then
\begin{equation}\label{no_incl_D1}\hat{f}(im(\mG,\mA))\not\subseteq im(\mG,\mB).\end{equation}
\end{theorem}

\subsection{General observations}\label{res_obs}
We start with an observation for states where the preimage of $f$ is a singleton.
\begin{lemma}\label{lemma_preimage_singleton}
Consider a BN $\mB\mN=(\mG,\PP)$ with random vector $(X_v)_{v\in V}$, $f:\mA\to \mB$, and $(U_v)_{v\in V}:=(f(X_v))_{v\in V}$. Assume that for $b_{pa(v)}\in \mB^{pa(v)}$, $f^{-1}(b_{pa(v)})=a_{pa(v)}$ with $a_{pa(v)}\in \mA^{pa(v)}$. 
Then, for all $b_{nd^*(v)}\in \mB^{nd^*(v)}$ with $b_{nd^*(v)}|_{pa(v)}=b_{pa(v)}$ and $\PP(U_{nd(v)}=b_{nd(v)})>0$ we have the following
\begin{equation}\label{equ_007}
\PP(U_v=b_v|U_{nd(v)}=b_{nd(v)})=\PP(U_v=b_v|U_{pa(v)}=b_{pa(v)}).
\end{equation}
\end{lemma}
\begin{proof}
Assume that $b_{pa(v)}\in \mB^{pa(v)}$ is such that $f^{-1}(b_{pa(v)})=a_{pa(v)}$ with $a_{pa(v)}\in \mA^{pa(v)}$. Let $b_{nd^*(v)}\in \mB^{nd^*(v)}$ with $b_{nd^*(v)}|_{pa(v)}=b_{pa(v)}$ and $\PP(U_{nd(v)}=b_{nd(v)})>0$.
By definition it is enough to show that \eqref{equ_007} holds. Hence, the following equality will be enough
$$
\PP(U_{nd^*(v)}=b_{nd^*(v)})=\PP(U_v=b_v|U_{pa(v)}=b_{pa(v)})\PP(U_{nd(v)}=b_{nd(v)}).
$$
By the definition of the factorisation \eqref{factorisation_BN}, we rewrite
$
\PP(U_{nd^*(v)}=b_{nd^*(v)})$ as 

$$\PP(U_{nd^*(v)}=b_{nd^*(v)})=\sum_{\tilde{a}_{nd^*(v)}\in f^{-1}(b_{nd^*(v)})}\PP(X_{nd^*(v)}=\tilde{a}_{nd^*(v)})$$
$$=\sum_{\tilde{a}_{nd^*(v)}\in f^{-1}(b_{nd^*(v)})}\prod_{\tilde{v}\in nd^*(v)}\PP(X_{\tilde{v}}=\tilde{a}_{\tilde{v}}|X_{pa(\tilde{v})}=\tilde{a}_{pa(\tilde{v})}).$$
Now we use $f^{-1}(b_{pa(v)})=a_{pa(v)}$, and we rewrite the above as 
$$=\sum_{\tilde{a}_{v}\in f^{-1}(b_{v})}\sum_{\tilde{a}_{nd(v)}\in f^{-1}(b_{nd(v)})}\PP(X_v=\tilde{a}_{v}|X_{pa(v)}=a_{pa(v)})\prod_{\tilde{v}\in nd(v)}\PP(X_{\tilde{v}}=\tilde{a}_{\tilde{v}}|X_{pa(\tilde{v})}=\tilde{a}_{pa(\tilde{v})})$$
$$=\sum_{\tilde{a}_{v}\in f^{-1}(b_{v})}\PP(X_v=\tilde{a}_{v}|X_{pa(v)}=a_{pa(v)})\sum_{\tilde{a}_{nd(v)}\in f^{-1}(b_{nd(v)})}\prod_{\tilde{v}\in nd(v)}\PP(X_{\tilde{v}}=\tilde{a}_{\tilde{v}}|X_{pa(\tilde{v})}=\tilde{a}_{pa(\tilde{v})}).$$ We can again use the factorisation \eqref{factorisation_BN} together with the equality of events between $\{X_{pa(v)}=a_{pa(v)}\}$ and $\{U_{pa(v)}=b_{pa(v)}\}$ to rewrite the above as
$$
=\PP(U_v=b_v|U_{pa(v)}=b_{pa(v)})\PP(U_{nd(v)}=b_{nd(v)})
$$
\end{proof}

\begin{lemma}\label{prepare_suff_thm_D2}
\color{black}
Consider a BN $\mB\mN=(\mG,\PP)$ with random vector $(X_v)_{v\in V}$, $f:\mA\to \mB$, and $(U_v)_{v\in V}:=(f(X_v))_{v\in V}$. Assume the following holds for a $v\in V_p$ with $depth(v)\geq1$ for$v$ fix.

For all 
$b_{pa(v)}\in B^{pa(v)}$, and all $a_{pa(\tilde{v})}\in A^{pa(\tilde{v})}$ for $\tilde{v}\in pa(v)$, the following holds.

\begin{equation}\label{eq:nec_cond}
\begin{split}
 \PP(U_{pa(v)}=b_{pa(v)}|X_{pa^2(v)}=a_{pa^2(v)})\PP(U_v=b_v|U_{pa(v)}=b_{pa(v)})=\\
\sum_{a_{pa(v)}\in f^{-1}(b_{pa(v)})} \PP(X_{pa(v)}=a_{pa(v)}|X_{pa^2(v)}=a_{pa^2(v)})\PP(U_v=b_v|X_{pa(v)}=a_{pa(v)}).
\end{split}
 \end{equation}
Then, 
\begin{enumerate}
\item For any $b_{pr(v)}\in B^{pr(v)}$ with $\PP(U_{pr(v)}=b_{pr(v)})>0$ the following holds:
\begin{equation}\label{eq:nec_cond_cons}
\PP(U_{v}=b_{v}|U_{pr(v)}=b_{pr(v)})=\PP(U_v=b_v|U_{pa(v)}=b_{pa(v)}).
\end{equation}
\item 
If the direct descendants of $pa(v)$ are $v$, i.e. $dde(pa(v))=v$, and either $depth(v)=1$ or $depth(v)\geq2$ and the direct descendants of $pa^2(v)$ are a subset of $pa(v)$, i.e. $dde(pa^2(v))\subseteq pa(v)$, then for any $b_{nd(v)}\in B^{nd(v)}$ such that $\PP(U_{nd(v)}=b_{pr(v)})>0$ the following holds:
\begin{equation}\label{eq:sec_cond_cons}
\PP(U_{v}=b_{v}|U_{nd(v)}=b_{nd(v)})=\PP(U_v=b_v|U_{pa(v)}=b_{pa(v)}).
\end{equation}
\end{enumerate}
\end{lemma}
\begin{proof}
\begin{enumerate}
\item If $depth(v)=1$, then $pr(v)=pa(v)$, hence \eqref{eq:nec_cond_cons} follows by definition. Therefore, assume $v$ has $depth(v)\geq 2$, and is such that for all 
$b_{pa(v)}\in B^{pa(v)}$, and all $a_{pa(\tilde{v})}\in A^{pa(\tilde{v})}$ for $\tilde{v}\in pa(v)$, \eqref{eq:nec_cond} holds.

Consider $b_{pr(v)}\in B^{pr(v)}$ such that $\PP(U_{pr(v)}=b_{pr(v)})>0$. Then, 
$$
\PP(U_{pr^*(v)}=b_{pr^*(v)})=\PP(U_v=b_v,U_{pa(v)}=b_{pa(v)},U_{pa^2(v)}=b_{pa^2(v)},U_{W}=b_{W}),
$$
with $W=pr^*(v)\setminus (pa^*(v)\cup pa^2(v))$. Using the law of total probability and the factorisation property of the BN we can write it as
\begin{equation*}
\begin{multlined}
\sum_{a_{pa^2(v)}\in f^{-1}(b_{pa^2(v)})}\sum_{a_{pa(v)}\in f^{-1}(b_{pa(v)})}\PP(U_v=b_v|X_{pa(v)}=a_{pa(v)})\\
\cdot \PP(X_{pa(v)}=a_{pa(v)}|X_{pa^2(v)}=a_{pa^2(v)})\PP(X_{pa^2(v)}=a_{pa^2(v)},U_W=b_W),
\end{multlined}
 \end{equation*}
which using \eqref{eq:nec_cond} equals
\begin{equation*}
\begin{multlined}
\sum_{a_{pa^2(v)}\in f^{-1}(b_{pa^2(v)})}\PP(U_v=b_v|U_{pa(v)}=b_{pa(v)})\PP(U_{pa(v)}=b_{pa(v)}|X_{pa^2(v)}=a_{pa^2(v)})\\
\cdot \PP(X_{pa^2(v)}=a_{pa^2(v)},U_W=b_W).
\end{multlined}
 \end{equation*}
After factoring out $\PP(U_v=b_v|U_{pa(v)}=b_{pa(v)})$, we get
$$
\PP(U_v=b_v|U_{pa(v)}=b_{pa(v)})\cdot \sum_{a_{pa^2(v)}\in f^{-1}(b_{pa^2(v)})}\PP(U_{pa(v)}=b_{pa(v)}|X_{pa^2(v)}=a_{pa^2(v)})\PP(X_{pa^2(v)}=a_{pa^2(v)},U_W=b_W),
$$
which using that $(X_v)_{v\in V}$ is a BN equals
$$
\PP(U_v=b_v|U_{pa(v)}=b_{pa(v)})\cdot\PP(U_{pa(v)}=b_{pa(v)},U_{pa^2(v)}=b_{pa^2(v)},U_W=b_W).
$$
Hence \eqref{eq:nec_cond_cons} holds, which is what we wanted to show. 
\item Assume first that $v$ has $depth(v)\geq 2$, satisfies $dde(pa(v))=v$ and $dde(pa^2(v))\subseteq pa(v)$, and is such that for all 
$b_{pa(v)}\in B^{pa(v)}$, and all $a_{pa(\tilde{v})}\in A^{pa(\tilde{v})}$ for $\tilde{v}\in pa(v)$, \eqref{eq:nec_cond} holds.

While the proof is essentially the same as for (1), we provide it for completeness. Consider $b_{nd(v)}\in B^{nd(v)}$ such that $\PP(U_{nd(v)}=b_{nd(v)})>0$. Then, 
$$
\PP(U_{nd^*(v)}=b_{nd^*(v)})=\PP(U_v=b_v,U_{pa(v)}=b_{pa(v)},U_{pa^2(v)}=b_{pa^2(v)},U_{W}=b_{W})
$$
with $W=nd(v)\setminus (pa(v)\cup pa^2(v))$. Using the law of total probability and $dde(pa(v))=v$ and $dde(pa^2(v))\subseteq pa(v)$, which implies  that $pa(v)$ d-separates $v,pa^2(v)$ as well as that $pa^2(v)$ d-separates $pa(v),W$ leads to the following CPDs according to Theorem \ref{equivalence_local_MP}
\begin{equation*}
\begin{multlined}
\sum_{a_{pa^2(v)}\in f^{-1}(b_{pa^2(v)})}\sum_{a_{pa(v)}\in f^{-1}(b_{pa(v)})}\PP(U_v=b_v|X_{pa(v)}=a_{pa(v)})\\
\cdot \PP(X_{pa(v)}=a_{pa(v)}|X_{pa^2(v)}=a_{pa^2(v)})\PP(X_{pa^2(v)}=a_{pa^2(v)},U_W=b_W),
\end{multlined}
 \end{equation*}
which by using \eqref{eq:nec_cond} equals
\begin{equation*}
\begin{multlined}
\sum_{a_{pa^2(v)}\in f^{-1}(b_{pa^2(v)})}\PP(U_v=b_v|U_{pa(v)}=b_{pa(v)})\PP(U_{pa(v)}=b_{pa(v)}|X_{pa^2(v)}=a_{pa^2(v)})\\
\cdot \PP(X_{pa^2(v)}=a_{pa^2(v)},U_W=b_W).
\end{multlined}
 \end{equation*}
The rest of the proof is as in (1).\\
The case of $depth(v)=1$ is similar as \begin{itemize}
\item$v$ d-separates $pa(v)$ and $nd(v)\setminus pa(v)$ and
\item $X_{pa(v)}\perp\!\!\!\perp X_{nd(v)\setminus pa(v)}$ by the local Markov property of Theorem \ref{equivalence_local_MP} as $pa(v)$ are root nodes
\end{itemize}
 if the second set is nonempty. If $nd(v)\setminus pa(v)$ is empty, then $pa(v)=nd(v)$, i.e., \eqref{eq:sec_cond_cons} holds by definition.
 \end{enumerate}
\end{proof}
\color{black}

\subsection{Results for \ref{D1}}\label{res_D1}
In Theorem \ref{nec_D1} we give a necessary condition on the CPDs to factorise, while in Theorem \ref{char_D1} we characterise \ref{D1} via a condition on fractions of CPDs.

\begin{theorem}\label{nec_D1}
Consider a BN $\mB\mN=(\mG,\PP)$ with random vector $(X_v)_{v\in V}$, a state projection $f:\mA\to \mB$ and $(U_v)_{v\in V}:=(f(X_v))_{v\in V}$, and assume $(\mB\mN,f)$ satisfies \ref{D1}. Then, for all $v\in V$ of depth bigger than one, all 
$b_{pa(v)}\in B^{pa(v)},b_{pa^2(v)}\in B^{pa^2(v)}$ with $\PP(U_{pa(v)}=b_{pa(v)})\PP(U_{pa^2(v)}=b_{pa^2(v)})>0$ the following holds.

\begin{equation}\label{eq:nec_cond_}
\begin{split}
\PP(U_v=b_v|U_{pa(v)}=b_{pa(v)})\cdot \prod_{\tilde{v}\in pa(v)} \PP(U_{\tilde{v}}=b_{\tilde{v}}|U_{pa(\tilde{v})}=b_{pa(\tilde{v})})=\\ 
\sum_{a_{pa(v)}\in f^{-1}(b_{pa(v)})}\PP(U_v= b_v|X_{pa(v)}=a_{pa(v)})\cdot \PP(X_{pa(\tilde{v})}=a_{pa(v)}|U_{pa^2(v)}=b_{pa^2(v)}).
\end{split}
 \end{equation}
\end{theorem}
\begin{proof}\
By \ref{D1}, $(U_v)_{v\in V}$ is a BN with DAG $\mG$. By assumption, both the left and the right-hand side of \eqref{eq:nec_cond_} are defined. We consider the expression $\PP(U_v=b_v,U_{pa(v)}=b_{pa(v)}|U_{pa^2(v)}=b_{pa^2(v)})$ and show that it is equal to both sides of \eqref{eq:nec_cond_} separately. 
\begin{itemize}

\item
First, we can write out the conditional probabilities such that
\begin{equation}\begin{split}
\PP(U_v=b_v,U_{pa(v)}=b_{pa(v)}|U_{pa^2(v)}=b_{pa^2(v)})=\\
\sum_{a_{pa(v)}\in f^{-1}(b_{pa(v)})}\PP(U_v=b_v,X_{pa(v)}=a_{pa(v)}|U_{pa^2(v)}=b_{pa^2(v)})\end{split} \end{equation}
where equality holds by the law of total probability. By using Lemma \ref{elementary_lem} we get
$$
=\sum_{a_{pa(v)}\in f^{-1}(b_{pa(v)})}\PP(X_v=f^{-1}(b_v)|X_{pa(v)}=a_{pa(v)},U_{pa^2(v)}=b_{pa^2(v)})\PP(X_{pa(v)}=a_{pa(v)}|U_{pa^2(v)}=b_{pa^2(v)}).
$$
Then, using that $(X_v)_{v\in V}$ is a BN we get the right-hand side of \eqref{eq:nec_cond_}.

\item
Going the other way, we have
\begin{equation}\begin{split}
\PP(U_v=b_v,U_{pa(v)}=b_{pa(v)}|U_{pa^2(v)}=b_{pa^2(v)})=\\
\PP(U_v=b_v|U_{pa(v)}=b_{pa(v)},U_{pa^2(v)}=b_{pa^2(v)})\cdot \PP(U_{pa(v)}=b_{pa(v)}|U_{pa^2(v)}=b_{pa^2(v)})=\\
\PP(U_v=b_v|U_{pa(v)}=b_{pa(v)})\cdot \PP(U_{pa(v)}=b_{pa(v)}|U_{pa^2(v)}=b_{pa^2(v)})
\end{split} \end{equation}
where we first use Lemma using Lemma \ref{elementary_lem}. Afterwards we use that $(U_v)_{v\in V}$ is a BN. Then, by the factorisation of BNs \eqref{factorisation_BN}, we have
\begin{equation*}\begin{split}
\PP(U_v=b_v|U_{pa(v)}=b_{pa(v)})\cdot \PP(U_{pa(v)}=b_{pa(v)}|U_{pa^2(v)}=b_{pa^2(v)})=\\
\prod_{\tilde{v}\in pa(v)}\PP(U_{\tilde{v}}=b_{\tilde{v}}|U_{pa(\tilde{v})}=b_{pa(\tilde{v})})\PP(U_v=b_v|U_{pa(v)}=b_{pa(v)}).\\
\end{split} \end{equation*}

\end{itemize}
\end{proof}

\begin{corollary}\label{suff_thm_D2}
Consider a BN $\mB\mN=(\mG,\PP)$ with random vector $(X_v)_{v\in V}$ where $\mG$ is such that for all $v\in V$:
\begin{itemize}
\item If $depth(v)\geq 1$, then $dde(pa(v))=v$.
\item If $depth(v)\geq 2$, then $dde(pa^2(v))\subseteq pa(v)$.
\end{itemize}Let $f:\mA\to \mB$, and $(U_v)_{v\in V}:=(f(X_v))_{v\in V}$. Assume the following holds for all $v\in V_p$.
For all 
$b_{pa(v)}\in B^{pa(v)}$, and all $a_{pa(\tilde{v})}\in A^{pa(\tilde{v})}$ for $\tilde{v}\in pa(v)$, the following holds.
\begin{equation}\label{suff_D1_equ}
\begin{split}
 \PP(U_{pa(v)}=b_{pa(v)}|X_{pa^2(v)}=a_{pa^2(v)})\PP(U_v=b_v|U_{pa(v)}=b_{pa(v)})=\\
\sum_{a_{pa(v)}\in f^{-1}(b_{pa(v)})} \PP(X_{pa(v)}=a_{pa(v)}|X_{pa^2(v)}=a_{pa^2(v)})\PP(U_v=b_v|X_{pa(v)}=a_{pa(v)})
\end{split}
 \end{equation}
 Then, $(\mB\mN,f)$ satisfies \ref{D1}.
\end{corollary}
\begin{proof}
%
By assumption on the DAG and Corollary \ref{proof_princ}
it is enough to show the following for all $v\in V$ and all $b_{nd^*(v)}\in B^{nd^*(v)}$ with
 $\PP(U_{nd(v)}=b_{nd(v)})>0$:
$$
\PP(U_{v}=b_{v}|U_{nd(v)}=b_{nd(v)})=\PP(U_v=b_v|U_{pa(v)}=b_{pa(v)}),
$$
which is equivalent to
$$
\frac{\PP(U_{nd^*(v)}=b_{nd^*(v)})}{\PP(U_{nd(v)}=b_{nd(v)})}=\frac{\PP(U_{pa^*(v)}=b_{pa^*(v)})}{\PP(U_{pa(v)}=b_{pa(v)})}.
$$
Next we go through a case-by-case analysis for the depth of $v\in V$.
\begin{itemize}
\item If $depth(v)=0$, $pa(v)=\empty$ and by the local Markov property of Theorem \ref{equivalence_local_MP} $X_{v}\perp\!\!\!\perp X_{nd(v)}$.
\item If $depth(v)\geq1$, it holds by Lemma \ref{prepare_suff_thm_D2}.
\end{itemize}

\end{proof}
Note that equation \eqref{suff_D1_equ} is not invariant under change of initial distribution.

The following essentially corresponds to a global check following Theorem \ref{equivalence_local_MP}. The proof is postponed to Appendix $\S$ \ref{proof_char_D1_1}

\begin{theorem}\label{char_D1}
Consider a BN $\mB\mN=(\mG,\PP)$ with random vector $(X_v)_{v\in V}$ and state projection $f:\mA\to \mB$, and let $(U_v)_{v\in V}:=(f(X_v))_{v\in V}$. Then
$(\mB\mN,f)$ satisfies \ref{D1} if and only if for all $v\in V$ and all $w_1,w_2\in\mB^{nd^*(v)}$ with $w_1|_{pa^*(v)}=w_2|_{pa^*(v)}$, for $u_1=w_1|_{nd(v)},u_2=w_2|_{nd(v)}$ 
we have
\begin{equation}
\label{fund_equ}
\PP(U_{nd^*(v)}=w_1)\cdot \PP(U_{nd(v)}=u_2)=
\PP(U_{nd^*(v)}=w_2)\cdot \PP(U_{nd(v)}=u_1)
\end{equation}

\end{theorem}
\color{black}
\subsection{Results for \ref{D2}}\label{res_D2}
We give a sufficient condition in Theorem \ref{thm_suff_D2_2} based on a nonzero pattern of the CPDs. Theorem \ref{suff_thm_D2} gives a sufficient condition for \ref{D2}, showing that if some equality holds over sums of CPDs at each node, then \ref{D2} is true. 

\begin{theorem}\label{thm_suff_D2_2}
\color{black}
Consider BN $\mB\mN=(\mG,\PP)$ with random vector $(X_v)_{v\in V}$ and full support, where vertices of $\mG$ have in- and out-degree at most one and a state projection $f:\mA\to \mB$, with $(U_v)_{v\in V}:=(f(X_v))_{v\in V}$.
Assume $b_1,\cdots b_r\in B$ are such that $|f^{-1}(b_i)|=1$, while for $b_{r+1},\cdots b_j\in B$ we assume that $|f^{-1}(b_i)|>1$. Let $B_1:=\{b_1,\cdots b_r\}$ and $B_2:=\{b_{r+1},\cdots b_j\}$ such that $B=B_1\cup B_2$. 

Suppose the BN satisfies the following:\\
For all $v\in V_p$ and all 
$b_v\in B_2^v$ there is exactly one $\tilde{b}\in B_2$ such that for all $\tilde{a}_{pa(v)}\in A^{pa(v)}\setminus (f^{-1}(\tilde{b}))^{pa(v)}$ we have the 
following
\begin{equation}\label{eq:zero}
\PP(X_v\in f^{-1}(b_v)|X_{pa(v)}=\tilde{a}_{pa(v)})=0.
 \end{equation}
Then, $(\mB\mN,f)$ satisfies \ref{D2}.
\end{theorem}

\begin{proof}
First we note that by assumption for all $v\in V$ in \eqref{eq:zero}  we can define a map that maps $b_v$ to $\tilde{b}$, which we denote by $\pi_v:B_2\to B_2$. Furthermore, by assumption on the DAG $\mG$ and Corollary \ref{proof_princ_cor}
it is enough to show the following for all $v\in V$ and all $b_{pr^*(v)}\in B^{pr^*(v)}$ with
 $\PP(U_{pr(v)}=b_{pr(v)})>0$:
$$
\PP(U_{v}=b_{v}|U_{pr(v)}=b_{pr(v)})=\PP(U_v=b_v|U_{pa(v)}=b_{pa(v)}).
$$
This is equivalent to
$$
\frac{\PP(U_{pr^*(v)}=b_{pr^*(v)})}{\PP(U_{pr(v)}=b_{pr(v)})}=\frac{\PP(U_{pa^*(v)}=b_{pa^*(v)})}{\PP(U_{pa(v)}=b_{pa(v)})}.
$$
We also note that by assumption $pa(v)$ is a singleton set or empty for all $v\in V$.

While we have to show the statement for $(U_v[\tilde\alpha])_{v\in V}$, we will start with no change in initial distribution and explain later why this proof generalises. We distinguish two cases:\\
1.) If $b_{pa(v)}\in B_1$, equality holds by Lemma \ref{lemma_preimage_singleton}.\\
2.) Assume $b_{pa(v)}\not \in B_1$, then if $\PP(U_{pr(v)}=b_{pr(v)})>0$, by assumption it is necessary that for any $\tilde{v}\in pr(v)$ and any $w\in pa(\tilde{v})$, $b_{pr(v)|w}=\pi_{\tilde{v}}(b_{pr(v)|\tilde{v}})$.

Let $V_1:=V_S\cap pr(v)$, $V_2:=pa(v)$, and $V_3:=pr(v)\setminus(V_1\cup V_2)$. Then, 
$\PP(U_{V_3}=b_{V_3},U_{V_2}=b_{V_2},U_{V_1}=b_{V_1})=\PP(U_{V_3}=b_{V_3},U_{V_2}=b_{V_2},U_{V_1}=b_{V_1}\text{ or }U_{V_1}\neq b_{V_1})=
\PP(U_{V_3}=b_{V_3},U_{V_2}=b_{V_2})$, as
$\PP(U_{V_2}=b_{V_2}|U_{V_1}\neq b_{V_1})=0$ by equation \eqref{eq:zero} for $\PP(U_{V_1}\neq b_{V_1})>0$. \\
Repeating the same argument along the path of the DAG, we get:\\
- $\PP(U_{pr(v)}=b_{pr(v)})=\PP(U_{pa(v)}=b_{pa(v)})$, \\
- $\PP(U_{pr^*(v)}=b_{pr^*(v)})=\PP(U_{v}=b_{v},U_{pa(v)}=b_{pa(v)})$,\\
-  $\PP(U_{pa^*(v)}=b_{pa^*(v)})=\PP(U_{v}=b_{v},U_{pa(v)}=b_{pa(v)})$.\\
Therefore in particular \eqref{fund_equ2} holds.

Now coming back to showing that the same holds for $(U_v[\tilde\alpha])_{v\in V}$, it is enough to see that equation \eqref{eq:zero} also holds for $(X_v[\tilde\alpha])_{v\in V}$. We are done by the proof given before.
\end{proof}

\color{black}
For the next result we need a notion where the BN behaves "badly" in a vertex, which we show cannot happen in the setting of \ref{D2} if the depth of the vertex is bigger than one.
\begin{definition}
Consider a BN $\mB\mN=(\mG,\PP)$ with random vector $(X_v)_{v\in V}$ with full support and its reduction $(U_v)_{v\in V}:=(f(X_v))_{v\in V}$ under $f:\mA\to \mB$.
We say $v\in V$ is a bad vertices of $(\mB\mN,f)$ if there are two initial distributions
 $\mu,\nu$ such that the following holds:\\
- The depth of $v$ in $ \mG$ is bigger than one.\\
- There are $b_v\in \mB, b_{pa(v)}\in \mB^{pa(v)}$, such that 
\begin{equation}
\label{eq_bad_ed}\PP_\mu(U_v=b_v|U_{pa(v)}=b_{pa(v)})\neq \PP_\nu(U_v=b_v|U_{pa(v)}=b_{pa(v)}),
\end{equation}

$$\PP_\mu(U_{pa(v)}=b_{pa(v)})>0,$$
$$\PP_\nu(U_{pa(v)}=b_{pa(v)})>0$$
\end{definition}
Interestingly the following holds, where we also note that the full support is not needed in the proof as it is enough to have the CPDs defined.
\begin{theorem}\label{bad_edge_theorem}

\color{black}
Assume the BN $\mB\mN=(\mG,\PP)$ with random vector $(X_v)_{v\in V}$ has full support, and that $(\mB\mN,f)$ satisfies \ref{D2}. Then, $\mG$ has no bad vertices for $(X_v)_{v\in V}$ and $f:\mA\to \mB$.
\end{theorem}
\begin{proof}
Assume that $(U_v)_{v\in V}$ satisfies \ref{D2} for any starting distribution and assume $v\in V$ is a bad vertex for $b_v\in \mB, b_{pa(v)}\in \mB^{pa(v)}$ with initial distributions $\mu,\nu$ satisfying \eqref{eq_bad_ed}.
Then we define
$$
\lambda=\mathbbm{1}_{Z=0}\cdot\mu+\mathbbm{1}_{Z=1}\cdot\nu
$$
with $Z$ a random variable corresponding to a fair coin toss.

By \ref{D2}, $(U_v)_{v\in V}$ has the memoryless property under $\PP_{\lambda}(\cdot)$, hence information from previous vertices in conditional probabilities is forgotten. By assumption $v$ has depth bigger than one, hence $pa(v)$
has depth at least one. Therefore
\begin{equation}
\label{eq_memoryless}
\begin{split}
\PP_{\lambda}(U_v=b_v|U_{pa(v)}=b_{pa(v)},Z=1)=\PP_{\lambda}(U_v=b_v|U_{pa(v)}=b_{pa(v)})\\
=\PP_{\lambda}(U_v=b_v|U_{pa(v)}=b_{pa(v)},Z=0)
\end{split}
\end{equation}
We will also repeatedly use that the following equality holds by definition of the conditional probability if $\PP(A\cap B\cap C)>0$
$$
\PP(A\cap B\cap C)=\PP(A| B\cap C)\cdot \PP(B|C)\cdot \PP(C).
$$
By definition of the conditional probability the left-hand side of equation \eqref{eq_memoryless} equals
\begin{equation}
\label{eq_memoryless1}
\begin{split}
\frac{\PP_{\lambda}(U_v=b_v,U_{pa(v)}=b_{pa(v)},Z=1)}{\PP_{\lambda}(U_{pa(v)}=b_{pa(v)},Z=1)}\\
=
\frac{2\PP_{\lambda}(Z=1|U_v=b_v,U_{pa(v)}=b_{pa(v)})\cdot \PP_{\lambda}(U_v=b_v|U_{pa(v)}=b_{pa(v)})\cdot\PP_{\lambda}(U_{pa(v)}=b_{pa(v)})}{\PP_{\nu}(U_{pa(v)}=b_{pa(v)})}.
\end{split}
\end{equation}
On the other hand, the same applies to the right-hand side of equation \eqref{eq_memoryless}, which hence equals
\begin{equation}\label{eq_memoryless2}
\begin{split}
\frac{\PP_{\lambda}(U_v=b_v,U_{pa(v)}=b_{pa(v)},Z=0)}{\PP_{\lambda}(U_{pa(v)}=b_{pa(v)},Z=0)}\\
=\frac{2\PP_{\lambda}(Z=0|U_v=b_v,U_{pa(v)}=b_{pa(v)})\cdot \PP_{\lambda}(U_v=b_v|U_{pa(v)}=b_{pa(v)})\cdot\PP_{\lambda}(U_{pa(v)}=b_{pa(v)})}{\PP_{\mu}(U_{pa(v)}=b_{pa(v)})}.
\end{split}
\end{equation}
Then, setting both second terms of equations \eqref{eq_memoryless1} and \eqref{eq_memoryless2} equal and dividing by 
$\PP_{\lambda}(U_v=b_v|U_{pa(v)}=b_{pa(v)})\cdot\PP_{\lambda}(U_{pa(v)}=b_{pa(v)})$
gives the following:
\begin{equation}\label{eq_memoryless3}
\begin{split}
\frac{\PP_{\lambda}(Z=1|U_v=b_v,U_{pa(v)}=b_{pa(v)})}{\PP_{\nu}(U_{pa(v)}=b_{pa(v)})}
=
\frac{\PP_{\lambda}(Z=0|U_v=b_v,U_{pa(v)}=b_{pa(v)})}{\PP_{\mu}(U_{pa(v)}=b_{pa(v)})}.
\end{split}
\end{equation}
We rewrite $\PP_{\lambda}(Z=1|U_v=b_v,U_{pa(v)}=b_{pa(v)})$ by the definition of conditional probability, 
 \begin{equation}\label{eq_memoryless4}
\begin{split}
\PP_{\lambda}(Z=1|U_v=b_v,U_{pa(v)}=b_{pa(v)})=
\frac{\PP_{\nu}(U_v=b_v,U_{pa(v)}=b_{pa(v)})}{2\PP_{\lambda}(U_v=b_v,U_{pa(v)}=b_{pa(v)})}\\
=\frac{\PP_{\nu}(U_{pa(v)}=b_{pa(v)})\cdot \PP_{\nu}(U_v=b_v|U_{pa(v)}=b_{pa(v)})}{2\PP_{\lambda}(U_v=b_v,U_{pa(v)}=b_{pa(v)})}.
\end{split}
\end{equation}
By the same argument we rewrite $\PP_{\lambda}(Z=0|U_v=b_v,U_{pa(v)}=b_{pa(v)})$ as
 \begin{equation}\label{eq_memoryless5}
\begin{split}
\PP_{\lambda}(Z=0|U_v=b_v,U_{pa(v)}=b_{pa(v)})=
\frac{\PP_{\mu}(U_{pa(v)}=b_{pa(v)})\cdot \PP_{\mu}(U_v=b_v|U_{pa(v)}=b_{pa(v)})}{2\PP_{\lambda}(U_v=b_v,U_{pa(v)}=b_{pa(v)})}
\end{split}
\end{equation}
Then replacing $\PP_{\lambda}(Z=0|U_v=b_v,U_{pa(v)}=b_{pa(v)}),\PP_{\lambda}(Z=1|U_v=b_v,U_{pa(v)}=b_{pa(v)})$ in equation \eqref{eq_memoryless3} by the final expressions in \eqref{eq_memoryless5} and \eqref{eq_memoryless4} gives the following, where we already shortened the fraction and multiplied by $2\PP_{\lambda}(U_v=b_v,U_{pa(v)}=b_{pa(v)})$
$$
\PP_{\nu}(U_v=b_v|U_{pa(v)}=b_{pa(v)})=\PP_{\mu}(U_v=b_v|U_{pa(v)}=b_{pa(v)}).
$$
This is a contradiction to the assumption, and $v$ can not be a bad vertex.
\end{proof}
\color{black}
\begin{corollary}
\color{black}
Consider a BN $\mB\mN=(\mG,\PP)$ with full support and random vector $(X_v)_{v\in V}$, and let $f:\mA\to \mB$ be a reduction function with $(U_v)_{v\in V}:=(f(X_v))_{v\in V}$.
Assume $(\mB\mN,f)$ satisfies \ref{D2}. Then, for any
$\alpha
\in\prod_{v\in V_s}\Delta^{\mA^v}$ and any $b_v\in \mB,b_{pa(v)}\in\mB^{pa(v)}$ with $\PP_\alpha(U_{pa(v)}=b_{pa(v)})\neq 0$, the following CPDs for vertices $v\in V$ with depth bigger than one
$$\PP_\alpha(U_v=b_v|U_{pa(v)}=b_{pa(v)})$$
equal $\PP(U_v=b_v|U_{pa(v)}=b_{pa(v)})$ .
\end{corollary}

\subsection{Results for \ref{D3}}\label{res_D3}
In Theorem \ref{main_suff_nec_thmD3} we give a characterisation of \ref{D3}. Corollary \ref{main_suff_thm} gives the analogoue for BNs to the classical Kemeny-Snell condition for DTMCs \cite[Theorem 6.3.2]{kemenyfinite}, while in Theorem \ref{main_nec_thm_D3} we show that the condition from Corollary \ref{main_suff_thm} has to be satisfied for at least the edges that are connected to the \red{source nodes} of the DAG.

\begin{theorem}\label{main_suff_nec_thmD3}

\color{black}
Consider a BN $\mB\mN=(\mG,\PP)$ with full support and random vector $(X_v)_{v\in V}$, and let $f:\mA\to \mB$ with $(U_v)_{v\in V}:=(f(X_v))_{v\in V}$.
The following are equivalent:
\begin{enumerate}[label=\emph{(K\arabic*)}]
\item\label{K1} For all $v\in V$, the following holds, where for simplicity we denote $I=nd^*(v)$.
\begin{enumerate}[label=\emph{(C \Roman*)}]

\item Let $w_1,w_2\in\mB^{I}$ be two elements with $w_1|_{pa^*(v)}=w_2|_{pa^*(v)}$, and $a^1_{V_s},a^2_{V_s}$ be two arbitrary initial states such that if we denote $u_1=w_1|_{nd(v)},u_2=w_2|_{nd(v)}$ we have 
\begin{equation}
\label{equ_char4}
\PP_{a^1_{V_s}}(U_{I}=w_1)\cdot \PP_{a^2_{V_s}}(U_{nd(v)}=u_2)=
\PP_{a^2_{V_s}}(U_{I}=w_2)\cdot \PP_{a^1_{V_s}}(U_{nd(v)}=u_1)
\end{equation}

\color{black}

\end{enumerate}

\item\label{K2} For all $v\in V$, the following holds, where for simplicity we denote $I=nd^*(v)$.
\begin{enumerate}[label=\emph{(E \Roman*)}]

\item Let $w_1,w_2\in\mB^{I}$ be two elements with $w_1|_{pa^*(v)}=w_2|_{pa^*(v)}$, and $\alpha_1,\alpha_2$ be two initial distributions and $u_1=w_1|_{nd(v)},u_2=w_2|_{nd(v)}$. Then
\begin{equation}
\label{equ_char3}
\PP_{\alpha_1}(U_{I}=w_1)\cdot \PP_{\alpha_2}(U_{nd(v)}=u_2)=
\PP_{\alpha_2}(U_{I}=w_2)\cdot \PP_{\alpha_1}(U_{nd(v)}=u_1)
\end{equation}

\color{black}

\end{enumerate}
\item \ref{D3} holds for $(\mB\mN,f)$.

\end{enumerate}

\end{theorem}
\begin{proof}
\ref{D3} $\implies$ \ref{K2}:\\
Let $w_1,w_2\in\mB^{I}$ be such that $w_1|_{pa^*(v)}=w_2|_{pa^*(v)}$ with $b_v=w_1|_{v},b_{pa(v)}=w_1|_{pa(v)}$, $\alpha_1,\alpha_2$ two arbitrary initial distributions, $u_1=w_1|_{nd(v)},u_2=w_2|_{nd(v)}$. If $\PP_{\alpha_1}(U_{nd(v)}=u_1)=0$ or $\PP_{\alpha_2}(U_{nd(v)}=u_2)=0$, then \eqref{equ_char3} holds. Hence, assume $\PP_{\alpha_1}(U_{nd(v)}=u_1)\PP_{\alpha_2}(U_{nd(v)}=u_2)>0$.

As \ref{D3} implies that $(U_v)_{v\in V}$ factorises both under $\PP_{\alpha_2}(\cdot),\PP_{\alpha_2}(\cdot)$, we have that 
$$\frac{\PP_{\alpha_1}(U_{I}=w_1)}{\PP_{\alpha_1}(U_{nd(v)}=u_1)}=\PP_{\alpha_1}(U_v=b_v|U_{pa(v)}=b_{pa(v)}),$$

$$\frac{\PP_{\alpha_2}(U_{I}=w_2)}{\PP_{\alpha_2}(U_{nd(v)}=u_2)}=\PP_{\alpha_2}(U_v=b_v|U_{pa(v)}=b_{pa(v)}),$$
and by \ref{D3} the CPDs have to be the same, and we can set the above to equations equal and multiply by $\PP_{\alpha_1}(U_{nd(v)}=u_1)\PP_{\alpha_2}(U_{nd(v)}=u_2)$ to get \eqref{equ_char3}.\\
\ref{K2} $\implies$ \ref{K1}:\\
Note that \eqref{equ_char4} follows from\eqref{equ_char3}. \\
\ref{K1} $\implies$ \ref{D3}:\\
Let $\alpha=\sum_{i=1}^jp_i\mathbbm{1}_{X_{V_s}=a^i_{V_s}}$ be an arbitrary initial distribution, and $v\in V$ and $b_{nd^*(v)}\in B^{nd^*(v)}$ arbitrary such that 
 $\PP_\alpha(U_{nd(v)}=b_{nd(v)})>0$. By definition and Lemma \ref{proof_princ} it is enough to show that
\begin{equation}
\label{eq_to_show_h}
\PP_\alpha(U_{v}=b_{v}|U_{nd(v)}=b_{nd(v)})=\PP_\alpha(U_v=b_v|U_{pa(v)}=b_{pa(v)})
\end{equation}
and that this does not depend on $\alpha$ in the sense that whenever we consider another arbitrary initial distribution $\tilde{\alpha}$ with $\PP_{\tilde{\alpha}}(U_{nd(v)}=b_{nd(v)})>0$, the corresponding value of the CPD is the same.

We first note that by definition of the probability operator, 
\begin{equation}
\label{way_sum_coeff_prob}
\PP_\alpha(U_{nd^*(v)}=b_{nd^*(v)})=\sum_{i=1}^jp_i\cdot \PP_{a^i_{V_s}}(U_{nd^*(v)}=b_{nd^*(v)})
\end{equation}
Then using the definition of the conditional probability on the right-hand side of \eqref{eq_to_show_h} and entering the corresponding sums as \eqref{way_sum_coeff_prob} in numerator and denominator gives
$$
\PP_\alpha(U_v=b_v|U_{pa(v)}=b_{pa(v)})=\frac{\sum_{i=1}^jp_i\cdot \PP_{a^i_{V_s}}(U_{pa^*(v)}=b_{pa^*(v)})}{\sum_{i=1}^jp_i\cdot \PP_{a^i_{V_s}}(U_{pa(v)}=b_{pa(v)})}
$$
To calculate the second fraction we take into account all possibilities for states with fixed values in the coordinates $pa^*(v)$, resp. $pa(v)$. Let 
$$B_1:=b_{pa^*(v)}\times B^{nd(v)\setminus pa(v)}, \quad\quad B_2:=b_{pa(v)}\times B^{nd(v)\setminus pa(v)}$$
\begin{equation}
\label{temp_eq_sum}
\frac{\sum_{i=1}^jp_i\cdot \PP_{a^i_{V_s}}(U_{pa^*(v)}=b_{pa^*(v)})}{\sum_{i=1}^jp_i\cdot \PP_{a^i_{V_s}}(U_{pa(v)}=b_{pa(v)})}=
\frac{\sum_{
b_1\in B_1
}\sum_{i=1}^jp_i\cdot \PP_{a^i_{V_s}}(U_{nd^*(v)}=b_1)}{\sum_{
b_2\in B_2
}\sum_{i=1}^jp_i\cdot \PP_{a^i_{V_s}}(U_{nd(v)}=
b_2)}
\end{equation}
Note that whenever the event $\{U_{nd(v)}=b_2\}$ has zero probability on $\PP_{a^i_{V_s}}(\cdot)$, so does $\{U_{nd^*(v)}=(b_v,b_2)\}$ on $\PP_{a^i_{V_s}}(\cdot)$, and that for any initial state $a^i_{V_s}$ we derive from \ref{K1}, \eqref{equ_char4} that whenever $\PP_{a^i_{V_s}}(U_{nd(v)}=u_1)>0,\PP_{a^i_{V_s}}(U_{nd(v)}=u_2)>0,$
\begin{equation}
\label{eq_xx}
\PP_{a^i_{V_s}}(U_{nd^*(v)}=w_1)=c\cdot \PP_{a^i_{V_s}}(U_{nd(v)}=u_1),
\end{equation}
where the constant $c\geq 0$ only depends on the 
coordinates $pa^*(v)$ and is the same for any other initial state $a^k_{V_s}$ with $\PP_{a^k_{V_s}}(U_{nd(v)}=u_1)>0,\PP_{a^k_{V_s}}(U_{nd(v)}=u_2)>0$. Hence, if we denote by 
$$N_{b,a^i_{V_s}}:=\{b_2\in B_2|\PP_{a^i_{V_s}}(U_{nd(v)}=b_2)>0\},$$
such that $n_{b,a^i_{V_s}}:=|N_{b,a^i_{V_s}}|$, we get that the right-hand side of equation
\eqref{temp_eq_sum} equals
 $$\frac{\sum_{
b_1\in B_1
}\sum_{i=1}^jp_i\cdot \PP_{a^i_{V_s}}(U_{nd^*(v)}=b_1)}{\sum_{b_2\in B_2}\sum_{i=1}^jp_i\cdot \PP_{a^i_{V_s}}(U_{nd(v)}=b_2)}=
\frac{\sum_{i=1}^j p_i\cdot n_{b,a^i_{V_s}}  \cdot c}{\sum_{i=1}^jp_i \cdot  n_{b,a^i_{V_s}}}=c.$$ 

To show that the left-hand side of equation \eqref{eq_to_show_h} is equal to $c$, we use the same argument by writing it with the definition of the conditional probability and 
using \eqref{way_sum_coeff_prob} in numerator and denominator to get
$$
\PP_\alpha(U_v=b_v|U_{nd(v)}=b_{nd(v)})=\frac{\sum_{i=1}^jp_i\cdot \PP_{a^i_{V_s}}(U_{nd^*(v)}=b_{nd^*(v)})}{\sum_{i=1}^jp_i\cdot \PP_{a^i_{V_s}}(U_{nd(v)}=b_{nd(v)})}
$$
Let $K:=\{1\leq k\leq j|\PP_{a^i_{V_s}}(U_{nd(v)}=b_{nd(v)})>0\}$. Then, by definition the cardinality of $K$, $|K|$ is bigger than zero, and again by \eqref{eq_xx} and its argument we get that the left-hand side of equation \eqref{eq_to_show_h} equals
$$
\frac{\sum_{i\in K}p_i\cdot c}{\sum_{i\in K}p_i }=c
$$

\end{proof}

\begin{theorem}\label{main_suff_thm}
Consider a BN $\mB\mN=(\mG,\PP)$ with full support and random vector $(X_v)_{v\in V}$, and let $f:\mA\to \mB$ be a reduction function with $(U_v)_{v\in V}:=(f(X_v))_{v\in V}$.
Suppose the following holds:
For all $v\in V$, all $w,\tilde{w}\in \mA^{pa(v)}$
with $f(w)=f(\tilde{w})$, and all $b\in \mB$,
\begin{equation}\label{eq:fund1}
\PP(U_v=b|X_{pa(v)}=w)=\PP(U_v=b|X_{pa(v)}=\tilde{w}).
 \end{equation}
Then, $(\mB\mN ,f)$ satisfies \ref{D3}, and for all $v$ with depth greater or equal to one, all $w\in \mA^{pa(v)}$
 all $b\in \mB$, and all  $\alpha\in\prod_{v\in V_s}\Delta^{\mA^v}$ we have
$$
\PP_\alpha(U_v=b|U_{pa(v)}=f(w))=
\PP(U_v=b|X_{pa(v)}=w)
$$
\end{theorem}

\begin{proof}
By Theorem \ref{main_suff_nec_thmD3} its enough to show that for an arbitrary $v\in V$, where we denote $I=nd^*(v)$ and arbitrary $w_1,w_2\in\mB^{I}$ with $w_1|_{pa^*(v)}=w_2|_{pa^*(v)}$, $u_1=w_1|_{nd^*(v)},u_2=w_2|_{nd^*(v)}$ and $a_1,a_2\in\prod_{v\in V_s}\mA^v$ two arbitrary initial states we have

$$
\PP_{a_1}(U_{I}=w_1)\cdot \PP_{a_2}(U_{nd^*(v)}=u_2)=
\PP_{a_2}(U_{I}=w_2)\cdot \PP_{a_1}(U_{nd^*(v)}=u_1)
$$
If $\PP_{a_2}(U_{nd^*(v)}=u_2)$ or $\PP_{a_1}(U_{nd^*(v)}=u_1)$ are zero we are done as both sides of the equation are zero, hence assume both nonzero. If we prove the following claim we are done as $a_1,a_2$ are arbitrary.
\begin{claim}
$\PP_{a_1}(U_v=w_{1|v}|U_{nd(v)}=u_1)=\PP(U_v=w_{1|v}|X_{pa(v)}=a_{pa(v)})$, where $a_{pa(v)}$ is an arbitrary element of $f^{-1}(w_{1,pa(v)})$.
\end{claim}
We start by rewriting the left-hand side as follows, where $r:=I\setminus\{v\}$, $b_v=w_{1|v}$, $A_1:=A^{I\setminus V_s}\times a_1$
$$
\PP_{a_1}(U_v=b_v|U_{nd(v)}=u_1)=\frac{\PP_{a_1}(U_{I}=w_1)}{\PP_{a_1}(U_{nd(v)}=u_1)}$$
$$
=
\frac{\sum_{a_v\in f^{-1}(b_v)}\sum_{a_r\in f^{-1}(w_{1|r})\cap A_1}\PP_{a_1}(X_v=a_v|X_{pa(v)}=a_{r|pa(v)})\PP_{a_1}(X_{r}=a_{r})}{\sum_{a_r\in f^{-1}(w_{1|r})\cap A_1}\PP_{a_1}(X_{r}=a_{r})}
$$

$$
=
\frac{\sum_{a_r\in f^{-1}(w_{1|r})}(\sum_{a_v\in f^{-1}(b_v)}\PP_{a_1}(X_v=a_v|X_{pa(v)}=a_{r|pa(v)}))\PP_{a_1}(X_{r}=a_{r})}{\sum_{a_r\in f^{-1}(w_{1|r})}\PP_{a_1}(X_{r}=a_{r})}
$$

$$
=
\frac{\sum_{a_r\in f^{-1}(w_{1|r})}\PP_{a_1}(U_v=b_v|X_{pa(v)}=a_{r|pa(v)})\PP_{a_1}(X_{r}=a_{r})}{\sum_{a_r\in f^{-1}(w_{1|r})}\PP_{a_1}(X_{r}=a_{r})}
$$
Now we note that by assumption , $\PP_{a_1}(U_v=b_v|X_{pa(v)}=a_{r|pa(v)})$ equals $\PP(U_v=b_v|X_{pa(v)}=a_{r|pa(v)})$ and these are the same for any $a_{r|pa(v)}\in f^{-1}(w_{1,pa(v)})$. Hence, we use an arbitrary $a_{pa(v)}\in f^{-1}(w_{1,pa(v)})$ and factor it out to get
$$
=
\frac{\PP_{}(U_v=b_v|X_{pa(v)}=a_{pa(v)})\sum_{a_r\in f^{-1}(w_{1|r})}\PP_{a_1}(X_{r}=a_{r})}{\sum_{a_r\in f^{-1}(w_{1|r})}\PP_{a_1}(X_{r}=a_{r})}=\PP_{}(U_v=b_v|X_{pa(v)}=a_{pa(v)})
$$
which is what we wanted to show.

\end{proof}

While condition \eqref{eq:fund1} is not necessary as observed for NHDTMCs \cite{DEY2014}(or DTMCs \cite{kemenyfinite,gurvits,DEY2014}), the following holds.
Condition \eqref{eq:fund1} is necessary for all CPDs that connect to source nodes for \ref{D3} to hold. The argument is similar to the Kemeny-Snell condition for DTMCs in \cite{kemenyfinite}, so we postpone it to Appendix $\S$ \ref{proof_main_nec_thm_D3}.
\begin{theorem}\label{main_nec_thm_D3}
Consider a BN $\mB\mN=(\mG,\PP)$ with full support and random vector $(X_v)_{v\in V}$, and let $f:\mA\to \mB$ with $(U_v)_{v\in V}:=(f(X_v))_{v\in V}$.
Assume $(\mB\mN ,f)$ satisfies \ref{D3}. Then, for all vertices $v\in V$ of depth one
, all $w,\tilde{w}\in \mA^{pa(v)}$
with $f(w)=f(\tilde{w})$, and all $b\in \mB$,
$$
\PP(U_v=b|X_{pa(v)}=w)=\PP(U_v=b|X_{pa(v)}=\tilde{w}).
$$
\end{theorem}

\section{Examples and special cases}\label{sec_ex_cases}
\subsection{DTMCs}\label{sec_ex_cases_DTMCs}
Several sufficient conditions \cite{Rosenblatt_66,kemenyfinite} and an algebraic characterisation \cite[Theorem 2]{gurvits} have been studied for (weak) lumpability, which is a similar but different notion from \ref{D1}. As explained before, a DTMC satisfies \ref{D1}(the analog of weak lumpability) if and only if the lumped chain $(f(X_n))_{n\in\N}=(U_n)_{n\in\N}$ is a NHDTMC. 

For DTMCs, it is well-known that the reduced chain is a DTMC with the same transition probabilities independently of the initial condition if and only if the Kemeny-Snell condition holds, e.g., by \cite[Theorem 12]{gurvits}. 
We can strengthen this statement in our setting as we have that \ref{D3} holds for a DTMC $(X_n)_{n\in \N_0}$ if and only if its reduction $(f(X_n))_{n\in\N}=(U_n)_{n\in\N}$ is a NHDTMC for all initial distributions (see Lemma \ref{conn_Markov}). This follows as a corollary from Theorem \ref{main_nec_thm_D3}.
\begin{theorem}\label{DTMC_KS_strengthened}
 Consider a DTMC $(X_n)_{n\in \N_0}$ with stochastic matrix $P$ as a BN with reduction function $f:\mA\to \mB$. 
The following are equivalent:
\begin{enumerate}
\item \ref{D3} holds, i.e., $(f(X_n))_{n\in\N}=(U_n)_{n\in\N}$ is a NHDTMC for every initial distribution where the CPDs are independent of the initial distribution.
\item Condition \eqref{eq:fund1} holds for $P$ (i.e. the Kemeny Snell condition holds for the DTMC).

\end{enumerate}
Hence, the reduction is a DTMC and the CPDs do not depend on $n\in \N$.
\end{theorem}

While Theorem \ref{char_D1} still applies in the case of DTMCs, it does not give a useful characterisation of DTMCs whose lumped chain is a NHDTMC. We just give an application of Theorem \ref{char_D1} in an example to show
that a lumped chain of a DTMCs is not necessarily a NHDTMC.
\begin{example}
Consider the DTMC $(X_n)_{n\in \N_0}$ with states $\mA=\{a_1,a_2,a_3\}$ started in the initial state $a_1$, with the transition probabilities given by the following transition matrix.
$$
P = 
\begin{pmatrix}
1/2 & 1/4 & 1/4 \\
1/3 & 1/3 & 1/3 \\
0 & 1/2 & 1/2
\end{pmatrix}
$$
Consider the reduction map to $\mB=\{b_1,b_2\}$ given by $f(a_1)=f(a_2)=b_1,f(a_3)=b_2$. Then, by Theorem \ref{char_D1} the following equality must hold:
$$\PP(U_1=b_1,U_2=b_2,U_3=b_1)\cdot \PP(U_1=b_1,U_2=b_1,U_3=b_1,U_4=b_4)$$
$$=\PP(U_1=b_1,U_2=b_1,U_3=b_1)\cdot \PP(U_1=b_1,U_2=b_2,U_3=b_1,U_4=b_4).$$
However, the values of the above probabilities are the following:
$$\PP(U_1=b_1,U_2=b_2,U_3=b_1)= 1/8,\quad \PP(U_1=b_1,U_2=b_1,U_3=b_1,U_4=b_1)=(1/4)^2(50/9)$$
$$=\PP(U_1=b_1,U_2=b_1,U_3=b_1)=5/24 ,\quad \PP(U_1=b_1,U_2=b_2,U_3=b_1,U_4=b_1)=1/12.$$
Hence, the condition of Theorem \ref{char_D1} is not satisfied and $(U_n)_{n\in\N}$ is not a NHDTMC.
\end{example}

We also provide an example of a DTMC whose lumped chain is a NHDTMC but not a DTMC in the Appendix $\S$ \ref{examples_appendix}, example \ref{ex_not_DTMC_NHDTMC}.

\subsection{Other highly symmetric regular BNs}
Consider the special case of BNs where the nontrivial CPDs at each vertex $v\in V_p$ are the same. Such highly symmetric BNs include, e.g., DTMCs as special case, but the DAG structure could 
also be similar to

one of the following DAGs.
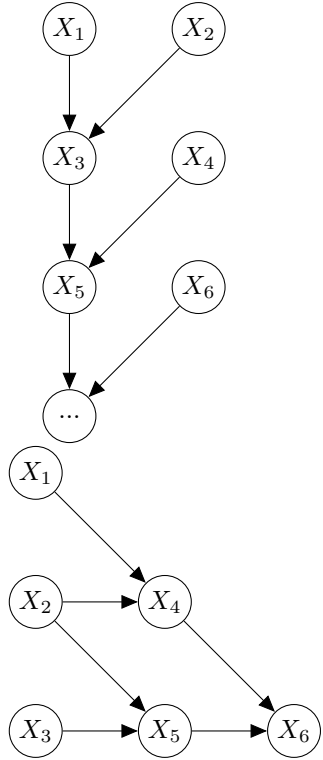
\begin{center}
\begin{tikzpicture}
    \node[latent] (theta) {$X_1$} ; %
    \node[latent, right=of theta] (z) {$X_2$} ; %
    \node[latent, below=of theta] (y) {$X_3$} ; %
    \node[latent, right=of y] (mu) {$X_4$} ; %
    \node[latent, below=of y] (sigma) {$X_5$} ; %
        \node[latent, right=of sigma] (sigm2) {$X_6$} ; %
        \node[latent, below=of sigma] (sigm3) {...} ; %
    \edge {theta} {y} ; %
        \edge {z} {y} ; %
    \edge {y} {sigma} ; %
            \edge {mu} {sigma} ; %
            \edge  {sigma}{sigm3} ; %
                              \edge {sigm2} {sigm3} ; %
\end{tikzpicture}
\qquad
\qquad

\begin{tikzpicture}
    \node[latent] (theta) {$X_1$} ; %
       \node[latent,  below=of theta] (t1) {$X_2$} ; %
 \node[latent,   right=of t1] (t2) {$X_4$} ; %
  \node[latent,   below=of t1] (t3) {$X_3$} ; %
   \node[latent,   right=of t3] (t4) {$X_5$} ; %
 \node[latent,   right=of t4] (z) {$X_6$} ; %
        \edge {theta} {t2} ; %
            \edge {t1} {t2} ; %
               \edge {t1} {t4} ; %
                           \edge {t3} {t4} ; %
   \edge {t4} {z} ; %
    \edge {t2} {z} ; %
\end{tikzpicture}
\end{center}
For such BNs, but also regular BNs with any number incoming edges fix the extension of Theorem \ref{DTMC_KS_strengthened} holds as follows.
\begin{corollary}
Consider a random vector $(X_v)_{v\in V}$ with full support as a BN $(\mG,\PP)$, where the DAG has only vertices with zero or $k$ incoming vertices, where $k\geq1$ is a fixed number, with all non-trivial CPDs the same. Assume there is at least one vertex of depth one. Consider a state projection $f:\mA\to \mB$, with $(U_v)_{v\in V}:=(f(X_v))_{v\in V}$ the projected random vector. Then, we have \ref{D3} if and only if \eqref{eq:fund1} holds.
\end{corollary}
\begin{proof}
Suppose we show that there is at least a vertex whose incoming edges come directly from initial distributions. In that case we are done by Theorem \ref{main_nec_thm_D3}, as this vertex then has degree one, and as all nontrivial CPDs are the same by assumption. This is clearly the case; hence we are done.

\end{proof}
By the same proof, the same statement still holds if we allow a BN with various kinds of CPDs, say $C_1,\cdots,C_k$, but any CPD we use in the construction of the BN appears at least once with only incoming edges from initial distributions.
\subsection{Non-homogeneous DTMCs}\label{subsec_NHDTMCs}
Non-homogeneous DTMCs (NHDTMCs) are determined by a sequence of potentially different stochastic matrices $(P_n)_{n\in \N_0}$. For NHDTMCs, \ref{D2} is not equivalent to \ref{D3} , e.g., by \cite[Example 1]{DEY2014}.  Sufficient conditions and special cases for necessity for \ref{D2} are studied \cite{DEY2014}. However, Theorem \ref{bad_edge_theorem} shows that in the case of \ref{D2} for NHDTMCs, the difference is actually marginal. Hence, for all $n\geq 2$, the CPDs of the reduced NHDTMC must be independent of the initial distribution. 
More formally concerning \ref{D2}, the following is a consequence of Theorem \ref{bad_edge_theorem}.
\begin{corollary}
Consider a NHDTMC $(X_n)_{n\in \N_0}$ with reduction function $f:\mA\to \mB$, and assume that $(U_n)_{n\in \N_0}:=(f(X_n))_{n\in \N_0}$ is a NHDTMC for all initial distributions (corresponding to \ref{D2}). Then, for any two initial distributions $\alpha_1,\alpha_2$ and any $b,\tilde{b}\in \mB$, and any $n\geq2$ the following CPDs are equal $$\PP_{\alpha_1}(U_n=b|U_{n-1}=\tilde{b})=\PP_{\alpha_2}(U_n=b|U_{n-1}=\tilde{b}).$$
\end{corollary}

We next outline the relations of some of our results to ibid.

 Corollary \ref{main_suff_thm} generalises the classical Kemeny-Snell condition \cite[Theorem 6.3.2]{kemenyfinite} as well as a sufficiency result for NHDTMCs \cite[Theorem 1]{DEY2014} for \ref{D3}. Theorem \ref{main_suff_nec_thmD3} generalises the previous and offers a characterisation of \ref{D3}.
Theorem \ref{nec_D1} gives a necessary condition on the CPDs generalising \cite[Theorem 3 (ii)]{DEY2014} for \ref{D1}. On the other hand, Theorem \ref{char_D1} characterises \ref{D1} for NHDTMCs but is not easy to check in practice.
Theorem \ref{thm_suff_D2_2}, assuming a nonzero pattern of the CPDs, and Theorem \ref{suff_thm_D2} generalise \cite[Theorem 3 (i)]{DEY2014} resp. \cite[Theorem 3 (iii)]{DEY2014}.

\section{Discussion}\label{sec_disc}
We studied a dimensionality reduction for discrete BNs that preserves conditional independence (CI) statements. Our findings encompass easily verifiable sufficient conditions for
 \ref{D1}, \ref{D2}, and \ref{D3} but also characterisations of the prior. We also established connections between our results and well-known models such as DTMCs, NHDTMCs, and other highly symmetric BNs.

It is easy to see that our results
extend to the situation of BNs $\tilde{\mB\mN}=(\mG,\tilde{\PP})$ with random vector $(\tilde{X}_v)_{v\in V}$ where each $\tilde{X}_v$ has a finite numbers of different states $\mA_v$ for vertices $v\in V$ under general reduction functions for
each vertex $v$. In this setting, $f_v:\mA^v\to\mB^v$ maps a different set of states $\mA^v$ with a different surjective maps for $v\in V$. One can tackle that situation by
defining a BN $\mB\mN=(\mG,\PP)$ with random vector $(X_v)_{v\in V}$ with state space $\mA:=\dot{\bigcup}_{v\in V}\mA_v$ and reduced space $\mB:=\dot{\bigcup}_{v\in V}\mB_v$ with obvious reduction function $f:\mA\to \mB$.
Therefore \ref{D1} holds for $(\mB\mN,f)$ if and only if it holds for $(\tilde{\mB\mN},(f_v)_{v\in V})$, whereas if \ref{D2}/\ref{D3} holds for $(\mB\mN,f)$, \ref{D2}/\ref{D3} also holds for $(\tilde{\mB\mN},(f_v)_{v\in V})$.

However, several questions and avenues for further exploration stay open. While we have focused on discrete BNs, the study of reductions in continuous state spaces for directed graphical models remains underdeveloped, even in the case of DTMCs \cite{Rosenblatt_66}. Then, there is no similar analytical definition and characterisation for lumpings of undirected graphical models, though there are both approximate \cite{Wainwright06} or exact \cite{order_red} approaches to other reductions. Additionally, practical considerations such as implementation of reduction checks, algorithms to identify reductions, and optimal reduction strategies have not been addressed.

Our characterisation of \ref{D1} via Theorem \ref{char_D1} enables a simple check that can be implemented, and so does our characterisation of \ref{D3} via Theorem \ref{main_suff_nec_thmD3}. Such a characterisation for \ref{D2} is missing.

As discussed in $\S$ \ref{sec_ex_cases}, it would be interesting to have a practical condition on the lumped chain of a DTMC is a NHDTMC.

Given that the space of BNs amenable to reductions is often smaller than the space of BNs factorizing with a directed acyclic graph (DAG), it could be valuable to develop a framework to select approximate reductions, i.e. as in \cite{Weinan_cx}. 

Another promising direction is to approach reductions from the perspective of algebraic statistics \cite{drton2009} and study e.g., the dimensions of the corresponding notions for the reductions (cf., e.g., example \ref{ex_abstract}), and characterise when all distributions of a BN can be reduced without losing the factorisation property (i.e. in the sense of Theorem \ref{many_non_D1s}).

\appendix
\section{Proofs}
\subsection{Auxiliary remarks for BNs}\label{proof_princ}
The fact that we only have to check CPDs on nonzero events is repeatedly used in proofs and restated for convenience (see  \cite[definition 2.3 and 2.4]{koller2009probabilistic} and \cite[Theorem 3.27]{lauritzen1996} resp. Theorem \ref{equivalence_local_MP}).
\begin{lemma}\label{proof_princ_lemma} A random vector $(U_v)_{v\in V}$ factorises w.r.t. some DAG $\mG$ (in the sense of \eqref{factorisation_BN}) if and only if for all $v\in V$ and all $b_{nd^*(v)}\in B^{nd^*(v)}$ with
 $\PP(U_{nd(v)}=b_{nd(v)})>0$, the following holds:
\begin{equation}\label{BASIC_EQU}
\PP(U_{v}=b_{v}|U_{nd(v)}=b_{nd(v)})=\PP(U_v=b_v|U_{pa(v)}=b_{pa(v)})
\end{equation}
\end{lemma}

Similarly the following holds.
\begin{corollary}\label{proof_princ_cor} Let the DAG $\mG$ be such that all in- and outdegrees are either one or zero. Then, a random vector $(U_v)_{v\in V}$ factorises w.r.t. $\mG$ (in the sense of  \eqref{factorisation_BN}) if and only if for all $v\in V$ and all $b_{pr*(v)}\in B^{pr*(v)}$ with
 $\PP(U_{pr(v)}=b_{pr(v)})>0$, the following holds:
\begin{equation}\label{BASIC_EQU_s}
\PP(U_{v}=b_{v}|U_{pr(v)}=b_{pr(v)})=\PP(U_v=b_v|U_{pa(v)}=b_{pa(v)})
\end{equation}
\end{corollary}

We collect some observations on d-separation.
\begin{lemma}\label{lem_d-sep}
Consider a DAG $\mG$ with a vertex $v$. We have\begin{itemize}
\item $pa(v)$ d-separates $v$ and $pa^2(v)$.
\item $pa(v)$ d-separates $v$ and $nd(v)\setminus pa(v)$.
\item $pa(v)$ d-separates $v$ and $nd(v)\setminus pr(v)$.
\end{itemize}
\end{lemma}

Furthermore, we will also repeatedly use that the following trivial equalities holds by definition of the conditional probability
\begin{lemma}Let $A,B, C$ be events. Then, if $\PP(A| B\cap C)>0$,
\begin{itemize}\label{elementary_lem}
\item
$
\PP(A\cap B\cap C)=\PP(A| B\cap C)\cdot \PP(B|C)\cdot \PP(C).
$
\item $
\PP(A\cap B| C)=\PP(A| B\cap C)\cdot \PP(B|C).
$
\end{itemize}
\end{lemma}

\subsection{Proof Theorem \ref{many_non_D1s}}\label{proof_many_non_D1s}
We start with some preparatory observations, and then give the proof.

The following Lemma holds by definition with notions introduced in $\S$ \ref{General_red_sect}.
\begin{lemma}\label{lem_Markov_eq}
Consider two connected DAGs $\mG_1$, $\mG_2$ that are Markov equivalent and a surjective reduction function $f:\mA\to \mB$ with $|\mA |>|\mB |>1$ and $\hat{f}(\cdot)$ the image measure. Then
\begin{itemize}
\item $im(\mG_1,\mA)=im(\mG_2,\mA)$ and $im(\mG_1,\mB)=im(\mG_2,\mB)$
\item $\hat{f}(im(\mG_1,\mA))=\hat{f}(im(\mG_2,\mA))$
\end{itemize}
\end{lemma}
The next lemma concerns the random variables and image measures following consecutive projections.
\begin{lemma}\label{lem_commute}
Consider a discrete random vector $(X_v)_{v\in V}$ with state space in $\mA^V$, $z\in \Delta^{\prod_{v\in V}\mA}$ the corresponding probability distribution, a surjective reduction function $f:\mA\to \mB$, and $W\subseteq V$ a subset. Denote by $f(\cdot)$ the coordinate-wise map with $f$, $\pi_W(\cdot)$ when we project from $\mA^V\to \mA^W$, and correspondingly $\hat{f}(\cdot),\hat{\pi_W}(\cdot)$ their image measure maps.
Then the following holds:
\begin{enumerate}
\item $\pi_W(f((X_v)_{v\in V}))=\pi_W((f(X_v)_{v\in V}))=f(\pi_W(X_v)_{v\in V}))=f((X_w)_{w\in W})$.
\item $\hat{f}(\hat{\pi_W}(z))=\hat{\pi_W}(\hat{f}(z))$, and this is the probability distribution of $f((X_w)_{w\in W})$.
\end{enumerate}
\end{lemma}
\begin{proof}
(1) Follows by definition.\\
(2) It is enough to see that for $b_W\in \mB^W$, the following equivalence between sets holds
$$
\pi_W^{-1}(f^{-1}(b_W))=f^{-1}(\pi_W^{-1}(b_W))
$$
and that then $\PP(f((X_w)_{w\in W})=b_W)=\PP((X_v)_{v\in V}\in \pi_W^{-1}(f^{-1}(b_W)))$.
\end{proof}

\begin{lemma}\label{lem_marginal}
Consider a BN $\mB\mN=(\mG,\PP)$ with random vector $(X_v)_{v\in V}$. Let $W\subseteq V$ be a subset such that the vertex-induced subgraph $\mG_W$
has three vertices, has one of the following two forms
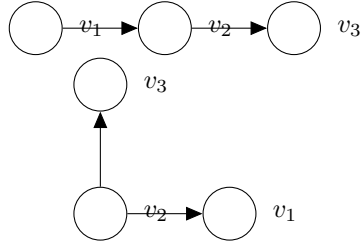
\begin{center}
\begin{tikzpicture}
$\mG_1 = $
    \node[latent] (theta) {$v_1$} ; %
    \node[latent, right=of theta] (z) {$v_2$} ; %
    \node[latent, right=of z] (y) {$v_3$} ; %
  
    \edge {theta} {z} ; %
      
    \edge {z} {y} ; %
\end{tikzpicture}

\begin{tikzpicture}
$\mG_2 = $
    \node[latent] (theta) {$v_2$} ; %
    \node[latent, right=of theta] (z) {$v_1$} ; %
    \node[latent, above=of theta] (y) {$v_3$} ; %
  
    \edge {theta} {z} ; %
      
    \edge {theta} {y} ; %
\end{tikzpicture}
\end{center}
and the vertex that is connected to the other two in $\mG_W$ has sum of in- and outdegrees in $\mG$ of two. Then, the following holds:
 \begin{enumerate}
 \item Both $(X_v)_{v\in V}$ and $(X_w)_{w\in W}$ satisfy the CI-condition $X_{v_3} \perp\!\!\!\perp X_{v_1}|X_{v_2}$.
 \item $im(\mG_W,\mA)\subseteq \hat{\pi_W}(im(\mG,\mA))$
 \end{enumerate}
\end{lemma}
\begin{proof}
We give the two arguments for completeness.
 \begin{enumerate}
 \item For $\mG_1,\mG_2$ it is clear. For $\mG$ by assumption the vertex that is connected to the other two in $\mG_W$ d-separates the other two vertices in $\mG$ by definition. Hence, the CI-statement follows by the global Markov property of Theorem \ref{equivalence_local_MP}
 \item To show that $im(\mG_W,\mA)\subseteq \hat{\pi_W}(im(\mG,\mA))$ we go through the two cases of DAGs, where we assume w.l.o.g. the indexing in $\mG$ corresponds to the indexes of $\mG_1,\mG_2$ above.
 \begin{itemize}
 \item Let $z_1\in im(\mG_1,\mA)$ be given as a factorisation of the three CPDs denoted $\PP_{\mG_1}(X_{v_1})$, $\PP_{\mG_1}(X_{v_2}|X_{v_1})$, and $\PP_{\mG_1}(X_{v_3}|X_{v_2})$. For $\mG$ we note that the marginal distribution in $v_1$ covers all distributions, i.e. $\hat{\pi_{v_1}}(im(\mG,\mA))=\Delta^\mA$, and as $v_2$ is not connected to other vertices in $\mG$ apart from $v_1,v_3$ we can take the same parameter in $\mG$ as in $\mG_1$ for the CPD $\PP(X_{v_2}|X_{v_1})$, i.e. $\PP_{\mG}(X_{v_2}|X_{v_1}):=\PP_{\mG_1}(X_{v_2}|X_{v_1})$. Concerning $v_3$, we can define the CPD in $\mG$ for this vertex as $\PP_{\mG}(X_{v_3}|X_{pa(v_3)}):=\PP_{\mG_1}(X_{v_3}|X_{v_2})$, which works as $v_2\in pa(v_3)$. Overall we have shown that then $z_1\in \hat{\pi_W}(im(\mG,\mA))$
 \item Let $z_1\in im(\mG_2,\mA)$ be given as a factorisation of the three CPDs denoted $\PP_{\mG_2}(X_{v_2})$, $\PP_{\mG_2}(X_{v_1}|X_{v_2})$, and $\PP_{\mG_2}(X_{v_3}|X_{v_2})$. Again by definition the marginal distribution of $\mG$ in $v_2$ covers all distributions,  i.e. $\hat{\pi_{v_2}}(im(\mG,\mA))=\Delta^\mA$. Then, we can choose the CPDs in $\mG$ for $v_1,v_3$ as  $\PP_{\mG}(X_{v_3}|X_{pa(v_3)}):=\PP_{\mG_2}(X_{v_3}|X_{v_2})$ and $\PP_{\mG}(X_{v_1}|X_{pa(v_1)}):=\PP_{\mG_2}(X_{v_1}|X_{v_2})$. Overall we have shown that then $z_1\in \hat{\pi_W}(im(\mG,\mA))$
 \end{itemize}
 \end{enumerate}
\end{proof}

Lastly, we give the actual proof of Theorem \ref{many_non_D1s}:
\begin{proof}
We will prove the statement by a sequence of reductions on the DAG $\mG$, the states $\mA,\mB$, and the map $f:\mA\to \mB$ as follows.
\begin{enumerate}
\item It is enough to prove \eqref{no_incl_D1} for an $f:\mA\to \mB$ with $|\mA|=|\mB|+1$ that maps exactly two states to the same state in $\mB$. This holds as any surjective reduction function $f$ can be written as a composition of  functions that map exactly two states to the same state.

\item It is enough to show \eqref{no_incl_D1} for an $f:\mA\to \mB$ that maps only $a_1,a_2$ together, i.e. $f(a_1)=f(a_2)$. This follows as any other function $\tilde{f}:\mA\to \mB$ that maps two states of $\mA$ together can be obtained via a composition with a bijection $g:\mA\to\mA$ such that $\tilde{f}=f\circ g$, and as $g$ is a bijection, i.e. $\hat{g}(im(\mG,\mA))=im(\mG,\mA)$.

\item\label{enum_2} 
By definition there is a connected vertex induced subgraph with three vertices with skeleton not equal to the complete graph nor a v.structure. Such a subgraph must be a DAG of one of the forms of $\mG_1,\mG_2$ as below.
\begin{itemize}
\item DAG $\mG_1$:
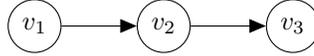
\begin{center}
\begin{tikzpicture}
    \node[latent] (theta) {$v_1$} ; %
    \node[latent, right=of theta] (z) {$v_2$} ; %
    \node[latent, right=of z] (y) {$v_3$} ; %
  
    \edge {theta} {z} ; %
      
    \edge {z} {y} ; %
\end{tikzpicture}
\end{center}
\item DAG $\mG_2$:

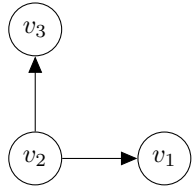
\begin{center}
\begin{tikzpicture}
    \node[latent] (theta) {$v_2$} ; %
    \node[latent, right=of theta] (z) {$v_1$} ; %
    \node[latent, above=of theta] (y) {$v_3$} ; %
  
    \edge {theta} {z} ; %
      
    \edge {theta} {y} ; %
\end{tikzpicture}
\end{center}

\end{itemize}

  We claim that it is enough to prove \eqref{no_incl_D1} for the DAGs $\mG_1$ and $\mG_2$ by contradiction, using Lemma \ref{lem_marginal} and Lemma \ref{lem_commute}. 
  
  The proof by contradiction is as follows. Assume \eqref{no_incl_D1} does not hold but we find (say) an element $z_1\in im(\mG_1,\mA)$ such that $\hat{f}(z_1)\not\in im(\mG_1,\mB)$. Then, by Lemma \ref{lem_marginal} (2) $z_1\in\hat{\pi_W}( im(\mG,\mA))$, hence let $z\in im(\mG,\mA)$ such that $\hat{\pi_W}(z)=z_1$. By assumption that \eqref{no_incl_D1} does not hold, $\hat{f}(z)$ factorises w.r.t. $\mG$. But then by Lemma \ref{lem_marginal} (1) $\hat{\pi_W}(\hat{f}(z))$ factorises w.r.t. $\mG_1$, and by Lemma \ref{lem_commute} (2) $\hat{\pi_W}(\hat{f}(z))=\hat{f}(\hat{\pi_W}(z))=\hat{f}(z_1)$, which is a contradiction.

\item Lastly it is enough to prove  \eqref{no_incl_D1} for the case where $|\mA|=3$ as we can consider the random vector to be on more coordinates by extension.
\end{enumerate}
Hence by (1), (2), and (4) we assume that $\mA=\{a_1,a_2,a_3\},\mB=\{b_1,b_2\}$ with $f(a_1)=f(a_2)=b_1,f(a_3)=b_2$. By (3) it is enough to prove \eqref{no_incl_D1}  for $\mG_1,\mG_2$. As $\mG_1,\mG_2$ have the same skeleton and have no immoralities, they are Markov equivalent by \cite[Theorem 3.8]{koller2009probabilistic}. Therefore $im(\mG_1,\mA)=im(\mG_2,\mA)$, $\hat{f}(im(\mG_1,\mA))=\hat{f}(im(\mG_2,\mA))$ and $im(\mG_1,\mB)=im(\mG_2,\mB)$ and by definition it is enough to give a BN of the form $\mG_1$ as defined above whose reductions does not satisfy the corresponding CI statement. This can be found in example \ref{ex_not_D1}.

\end{proof}
\color{black}

\subsection{Proof Theorem \ref{main_nec_thm_D3}}\label{proof_main_nec_thm_D3}
\begin{proof}
Let $v\in V$ have depth one, with $v:=v_0$, $I=pa(v)=\{v_1,\cdots , v_l\}$, and $b_0,b_1,\cdots b_l\in\mB$ such that $A_0=f^{-1}(b_0),\cdots ,A_l=f^{-1}(b_l)$. By \ref{D3}, the following defines a CPD:
\begin{equation}\label{equ_cpd}\PP(U_{v_0}=b_0|U_{I}=b_{I})=\frac{\PP(U_{v_0}=b_0,U_{v_1}=b_1,\cdots, U_{v_l}=b_l)}{\PP(U_{v_1}=b_1,\cdots, U_{v_l}=b_l)}\end{equation}
Then according to the factorisation of BNs \eqref{factorisation_BN}, we write the CPD of \eqref{equ_cpd} as
$$
\frac{\sum_{a_0\in A_0, \cdots, a_l\in A_l}\PP(X_{v_0}=a_0,\cdots,X_{v_l}=a_l)}{\sum_{a_1\in A_1,\cdots,a_l\in A_l}\PP(X_{v_1}=a_1,\cdots,X_{v_l}=a_l)}
=
\frac{\sum_{a_0\in A_0,\cdots,a_l\in A_l}\PP(X_{v_0}=a_0,\cdots,X_{v_l}=a_l)}{\sum_{a_1\in A_1,\cdots ,a_l\in A_l}\prod_{i=1}^l\alpha_i(a_k)}
$$
$$
=
\frac{\sum_{a_0\in A_0,\cdots,a_l\in A_l}\PP(X_{v_0}=a_0|X_{v_1}=a_1\cdots,X_{v_l}=a_l)\prod_{i=1}^l\alpha_i(a_i)}{\sum_{a_1\in A_1,\cdots,a_l\in A_l}\prod_{i=1}^l\alpha_i(a_i)}
$$
By \ref{D3}, fraction \eqref{equ_cpd} doesn't change value even if we change the initial distribution to another $(\tilde{\alpha}_v)_{v\in I}\in\Delta_{\prod_{v\in I}\mA^v}$ for the corresponding random variables $(U_v[\tilde{\alpha}])_{v\in I}$.
Hence if we choose all $\alpha_{v_i}=\mathbbm{1}_{a_j}$ where $a_j\in A_i$ for all $i\in \{0,\cdots l\}$, equation \eqref{eq:fund1} holds by the resulting equalities for the $v\in V$ in the proof.

\end{proof}
\subsection{Proof Theorem \ref{char_D1}}\label{proof_char_D1_1}
\begin{proof}

$\implies$: 
This follows by Lemma \ref{proof_princ_lemma}.\\
$\impliedby$: Assume the condition of \eqref{fund_equ} is true. Then, by Lemma \ref{proof_princ_lemma}  it is enough if for an arbitrary $v\in V$ and $b_{nd(v)}\in B^{nd(v)}$ where $\PP(U_{nd(v)}=b_{nd(v)})>0$, we  verify that the following equality holds

\begin{equation}
\label{fund_equ2}
\frac{\PP(U_{nd^*(v)}=b_{nd^*(v)})}{\PP(U_{nd(v)}=b_{nd(v)})}=
\frac{\PP(U_{pa^*(v)}=b_{pa^*(v)})}{\PP(U_{pa(v)}=b_{pa(v)})}.
\end{equation}
We calculate the second fraction by taking into account all possibilities for states with particular values in the coordinates $pa^*(v)$, resp. $pa(v)$. Let 
$$B_1:=b_{pa^*(v)}\times B^{nd(v)\setminus pa(v)}, \quad\quad B_2:=b_{pa(v)}\times B^{nd(v)\setminus pa(v)}$$
\begin{equation}
\label{fund_equ3}
\frac{\PP(U_{pa^*(v)}=b_{pa^*(v)})}{\PP(U_{pa(v)}=b_{pa(v)})}=\frac{\sum_{
b_1\in B_1
}\PP(U_{nd^*(v)}=b_1)}{\sum_{
b_2\in B_2
}\PP(U_{nd(v)}=
b_2)}
\end{equation}
Note that whenever the event $\{U_{nd(v)}=b_2\}$ has zero probability, so does $\{U_{nd^*(v)}=(b_v,b_2)\}$, and that we can equivalently write \eqref{fund_equ} as 
\begin{equation}
\label{eq_x}
\PP(U_{nd^*(v)}=\tilde{b}_{nd^*(v)})=c\cdot \PP(U_{nd(v)}=\tilde{b}_{nd(v)}),
\end{equation}
where the constant $c\geq 0$ only depends on the 
coordinates $pa^*(v)$. Hence, if we denote by 
$$N_b:=\{b_2\in B_2|\PP(U_{nd(v)}=b_2)>0\},$$
such that $n_b:=|N_b|\geq 1$, we get that the right-hand side of equation
\eqref{fund_equ3} is equal to $\frac{n_b\cdot c}{n_b}=c$ for $c=\frac{\PP(U_{nd^*(v)}=b_{nd^*(v)})}{\PP(U_{nd(v)}=b_{nd(v)})}$, hence we are done.

\end{proof}
\color{black}

\section{Further examples} \label{examples_appendix}

\begin{example}
Consider the following BN $\mB\mN=(\mG,\PP)$ with DAG $\mG$
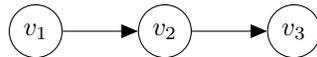
\begin{center}
\begin{tikzpicture}
    \node[latent] (theta) {$v_1$} ; %
    \node[latent, right=of theta] (z) {$v_2$} ; %
    \node[latent, right=of z] (y) {$v_3$} ; %
  
    \edge {theta} {z} ; %
      
    \edge {z} {y} ; %
\end{tikzpicture}
\end{center}
with $\mA=\{a_1,a_2,a_3\},\mB=\{b_1,b_2\}$, reduction function $f(a_1)=f(a_2)=b_1,f(a_3)=b_2$, initial distribution $(1/3,1/3,1/3)$ and CPDs all equal
$$
P = 
\begin{pmatrix}
1/2 & 1/4 & 1/4 \\
1/3 & 1/3 & 1/3 \\
0 & 1/2 & 1/2
\end{pmatrix}
.$$
By definition of \ref{D1}, if it holds the following equality between CPDs (also see example \ref{ex_abstract}) must hold:
$$
\PP(U_3=b_2|U_1=b_2,U_2=b_1)=\PP(U_3=b_2|U_2=b_1)
$$
We can write this equation as
$$
\frac{P_3(P_{3,1}P_{1,3}+P_{3,2}P_{2,3})}{P_3(P_{3,1}+P_{3,2})}=\frac{(P_1P_{1,1}+P_2P_{2,1}+P_3P_{3,1})P_{1,3}+(P_1P_{1,2}+P_2P_{2,2}+P_3P_{3,2})P_{2,3}}{(P_1P_{1,1}+P_2P_{2,1}+P_3P_{3,1})+(P_1P_{1,2}+P_2P_{2,2}+P_3P_{3,2})}
$$
which clearly does not hold with the values of $P$ given!
\end{example}

\begin{example}\label{ex_not_DTMC_NHDTMC}
Consider a DTMC with states $\mA=\{a_1,a_2,a_3,a_4\},\mB=\{b_1,b_2,b_3\}$, $f(a_i)=b_i$ for $i\neq 4$ and $f(a_4)=a_1$ with the following transition probabilities started in $a_1$
$$
P = 
\begin{pmatrix}
0 & 1/2 & 1/2 & 0\\
0 & 0 & 0& 1 \\
0 & 0 & 0& 1 \\
1 & 0 & 0& 0
\end{pmatrix}
$$
Then clearly $U_n$ is not a DTMC, but still a NHDTMC.
\end{example}

\begin{example}\label{ex_not_D1}
\color{black}

Consider the following BN $\mB\mN=(\mG,\PP)$ with DAG $\mG$
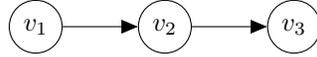
\begin{center}
\begin{tikzpicture}
    \node[latent] (theta) {$v_1$} ; %
    \node[latent, right=of theta] (z) {$v_2$} ; %
    \node[latent, right=of z] (y) {$v_3$} ; %
  
    \edge {theta} {z} ; %
      
    \edge {z} {y} ; %
\end{tikzpicture}
\end{center}
with $\mA=\{a_1,a_2,a_3\},\mB=\{b_1,b_2\}$, reduction function $f(a_1)=f(a_2)=b_1,f(a_3)=b_2$, initial distribution $(1/3,1/3,1/3)$ and CPDs all equal
$$
P = 
\begin{pmatrix}
1/2 & 1/4 & 1/4 \\
1/3 & 1/3 & 1/3 \\
0 & 1/2 & 1/2
\end{pmatrix}
.$$
By definition of \ref{D1}, if it holds the following equality between CPDs must hold:
$$
\PP(U_3=b_2|U_1=b_2,U_2=b_1)=\PP(U_3=b_2|U_2=b_1)
$$
We can write this equation as
$$
\frac{P_3(P_{3,1}P_{1,3}+P_{3,2}P_{2,3})}{P_3(P_{3,1}+P_{3,2})}=\frac{(P_1P_{1,1}+P_2P_{2,1}+P_3P_{3,1})P_{1,3}+(P_1P_{1,2}+P_2P_{2,2}+P_3P_{3,2})P_{2,3}}{(P_1P_{1,1}+P_2P_{2,1}+P_3P_{3,1})+(P_1P_{1,2}+P_2P_{2,2}+P_3P_{3,2})}
$$
which clearly does not hold with the values of $P$ given!
\end{example}

\begin{example}\label{ex_abstract}\label{ex_D1_not_D2}
Consider the following BN $\mB\mN=(\mG,\PP)$ with DAG $\mG$
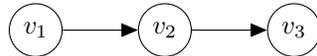
\begin{center}
\begin{tikzpicture}
    \node[latent] (theta) {$v_1$} ; %
    \node[latent, right=of theta] (z) {$v_2$} ; %
    \node[latent, right=of z] (y) {$v_3$} ; %
  
    \edge {theta} {z} ; %
      
    \edge {z} {y} ; %
\end{tikzpicture}
\end{center}
with $\mA=\{a_1,a_2,a_3\},\mB=\{b_1,b_2\}$, reduction function $f(a_1)=f(a_2)=b_1,f(a_3)=b_2$, initial distribution $(p^{v_1}_1,p^{v_1}_2,p^{v_1}_3)$ and CPD matrices
$$P^{v_2}=(p^{v_2}_{ij}),\quad P^{v_3}=(p^{v_3}_{ij}),
$$
i.e., e.g., for $1\leq i\leq 3,1\leq j\leq 3$ $p^{v_2}_{ij}=\PP(X_{v_2}=a_i|X_{v_1}=a_j)$.

By definition of \ref{D1}, there are 8 equations to check:
\begin{enumerate}
\item\label{it_1} $
\PP(U_3=b_1|U_2=b_1,U_1=b_1)=\PP(U_3=b_1|U_2=b_1)
$\item\label{it_2}  $
\PP(U_3=b_1|U_2=b_1,U_1=b_2)=\PP(U_3=b_1|U_2=b_1)
$\item\label{it_3}  $
\PP(U_3=b_1|U_2=b_2,U_1=b_1)=\PP(U_3=b_1|U_2=b_2)
$\item\label{it_4}  $
\PP(U_3=b_1|U_2=b_2,U_1=b_2)=\PP(U_3=b_1|U_2=b_2)
$\item\label{it_5}  $
\PP(U_3=b_2|U_2=b_1,U_1=b_1)=\PP(U_3=b_2|U_2=b_1)
$\item\label{it_6}  $
\PP(U_3=b_2|U_2=b_1,U_1=b_2)=\PP(U_3=b_2|U_2=b_1)
$\item\label{it_7}  $
\PP(U_3=b_2|U_2=b_2,U_1=b_1)=\PP(U_3=b_2|U_2=b_2)
$\item \label{it_8} $
\PP(U_3=b_2|U_2=b_2,U_1=b_2)=\PP(U_3=b_2|U_2=b_2)
$
\end{enumerate}
By Lemma \ref{lemma_preimage_singleton}, 4 of these are automatically satisfied as $f^{-1}(b_2)=a_3$, these are \eqref{it_3},\eqref{it_4}, \eqref{it_7} and \eqref{it_8}. Hence, we rewrite the equations for the remaining equations and their CPDs assuming that the conditional events are nonzero event (otherwise a fraction is not defined as a denominator is zero) for convenience. We note however that we can always reformulate 
\begin{equation}\label{label_eq}
\PP(A|B)=\PP(C|D)\text{ as }\PP(A\cap B)\PP(D)=\PP(B)\PP(C\cap D)
\end{equation}
 when we do not want to make a nonzero assumption.
\begin{itemize}
\item \eqref{it_1} $
\PP(U_3=b_1|U_2=b_1,U_1=b_1)=\PP(U_3=b_1|U_2=b_1)
$:
$$LHS=
\frac{P_1^{v1}(P_{1,1}^{v2}(P_{1,1}^{v3}+P_{1,2}^{v3})+P_{1,2}^{v2}(P_{2,1}^{v3}+P_{2,2}^{v3}))+P_2^{v1}(P_{2,1}^{v2}(P_{1,1}^{v3}+P_{1,2}^{v3})+P_{1,2}^{v2}(P_{2,1}^{v3}+P_{2,2}^{v3}))
}{P_1^{v1}(P_{1,1}^{v2}+P_{1,2}^{v2})+P_2^{v1}(P_{2,1}^{v2}+P_{2,2}^{v2})}$$
$$RHS=\frac{(P_1^{v1}P_{1,1}^{v2}+P_2^{v1}P_{2,1}^{v2}+P_3^{v1}P_{3,1}^{v2})(P_{1,1}^{v3}+P_{1,2}^{v3})
+
(P_1^{v1}P_{1,2}^{v2}+P_2^{v1}P_{2,2}^{v2}+P_3^{v1}P_{3,2}^{v2})(P_{2,1}^{v3}+P_{2,2}^{v3})
}{
P_1^{v1}(P_{1,1}^{v2}+P_{1,2}^{v2})+P_2^{v1}(P_{2,1}^{v2}+P_{2,2}^{v2})+P_3^{v1}(P_{3,1}^{v2}+P_{3,2}^{v2})}
$$

\item \eqref{it_2}
$
\PP(U_3=b_1|U_2=b_1,U_1=b_2)=\PP(U_3=b_1|U_2=b_1)
$:
$$LHS=
\frac{P_3^{v1}(P_{3,1}^{v2}(P_{1,1}^{v3}+P_{1,2}^{v3})+P_{3,2}^{v2}(P_{2,1}^{v3}+P_{2,2}^{v3}))
}{P_3^{v1}(P_{3,1}^{v2}+P_{3,2}^{v2})}$$
\begin{center}
RHS as above (i.e., as for \eqref{it_1})
\end{center}
\item \eqref{it_5}$
\PP(U_3=b_2|U_2=b_1,U_1=b_1)=\PP(U_3=b_2|U_2=b_1)
$:
$$LHS=
\frac{P_1^{v1}(P_{1,1}^{v2}P_{1,3}^{v3}+P_{1,2}^{v2}P_{2,3}^{v3})+P_2^{v1}(P_{2,1}^{v2}P_{1,3}^{v3}+P_{2,2}^{v2}P_{2,3}^{v3})
}{P_1^{v1}(P_{1,1}^{v2}+P_{1,2}^{v2})+P_2^{v1}(P_{2,1}^{v2}+P_{2,2}^{v2})}$$

$$RHS=
\frac{(P_1^{v1}P_{1,1}^{v2}+P_2^{v1}P_{2,1}^{v2}+P_3^{v1}P_{3,1}^{v2})P_{1,3}^{v3}+(P_1^{v1}P_{1,2}^{v2}+P_2^{v1}P_{2,2}^{v2}+P_3^{v1}P_{3,2}^{v2})P_{2,3}^{v3}}{(P_1^{v1}P_{1,1}^{v2}+P_2^{v1}P_{2,1}^{v2}+P_3^{v1}P_{3,1}^{v2})+(P_1^{v1}P_{1,2}^{v2}+P_2^{v1}P_{2,2}^{v2}+P_3^{v1}P_{3,2}^{v2})}$$
\item \eqref{it_6} $
\PP(U_3=b_2|U_2=b_1,U_1=b_2)=\PP(U_3=b_2|U_2=b_1)
$:
$$
LHS=\frac{P_3^{v1}(P_{3,1}^{v2}P_{1,3}^{v3}+P_{3,2}^{v2}P_{2,3}^{v3})}{P_3^{v1}(P_{3,1}^{v2}+P_{3,2}^{v2})}
$$
\begin{center}
RHS as above (i.e., as for \eqref{it_5})
\end{center}
Hence we have that the above equations in form of \eqref{label_eq} (i.e. without assuming nonzero) hold for a given BN $\mB\mN$ if and only if $(\mB\mN,f)$ satisfies \ref{D1}. Furthermore, for a given BN $\mB\mN$, $(\mB\mN,f) $satisfies \ref{D2} if and only if the corresponding equations hold for all $(p^{v_1}_1,p^{v_1}_2,p^{v_1}_3)\in \Delta_{\mA^{v_1}}$. Similar equations are available for \ref{D3} through Theorem \ref{main_suff_nec_thmD3}.
\end{itemize}

\end{example}

\begin{example}\label{mG_3}
Consider the following BN $\mB\mN=(\mG,\PP)$ with DAG $\mG$
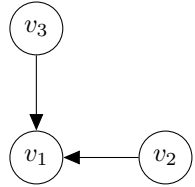
\begin{center}
\begin{tikzpicture}
    \node[latent] (theta) {$v_1$} ; %
    \node[latent, right=of theta] (z) {$v_2$} ; %
    \node[latent, above=of theta] (y) {$v_3$} ; %
  
    \edge {z} {theta} ; %
      
    \edge {y} {theta} ; %
\end{tikzpicture}
\end{center}
with $\mA=\{a_1,a_2,a_3\},\mB=\{b_1,b_2\}$, reduction function $f(a_1)=f(a_2)=b_1,f(a_3)=b_2$, initial distribution on $v_1$ and $v_3$ given. If the BN on $\mA$ has random vector $(X_1,X_2,X_3)$ then $X_2\perp\!\!\!\perp X_3$ and it normally loses independence when conditioned on $X_1$. It is easy to see that if $(U_1,U_2,U_3)$ with $U_i=f(X_i)$ then also $U_2\perp\!\!\!\perp U_3$, and these random vectors factorise with the same DAG.


\end{example}

 \bibliographystyle{plain}

 \bibliography{references} 

\end{document}